\documentclass[a4paper,oneside,reqno]{amsart}

\makeatletter
\g@addto@macro{\endabstract}{\@setabstract}
\newcommand{\authorfootnotes}{\renewcommand\thefootnote{\@fnsymbol\c@footnote}}%
\makeatother

\usepackage{url}

\usepackage[utf8]{inputenc}
\usepackage[T1]{fontenc}
\usepackage{lmodern}
\usepackage[english]{babel}
\usepackage{amsmath, amssymb, mathtools}
\usepackage{enumitem}
\usepackage{xfrac}

\renewcommand{\thefootnote}{\Roman{footnote}}

\usepackage[unicode=true, 
bookmarks=true,bookmarksopen=false, 
breaklinks=false,pdfborder={0 0 0},colorlinks=true] 
{hyperref}  

\usepackage{xcolor} 
\definecolor{cblue}{rgb}{0.16, 0.32, 0.75} 
\definecolor{cred}{rgb}{0.7, 0.11, 0.11}
\hypersetup{%  
	,linkcolor=cred  
	,citecolor=cblue
	,urlcolor=black
}

\usepackage{amsthm}

 \newtheorem{Theorem}{Theorem}[section]
 \newtheorem{Corollary}[Theorem]{Corollary}
 \newtheorem{Lemma}[Theorem]{Lemma}
 \newtheorem{Proposition}[Theorem]{Proposition}
 \theoremstyle{definition}
 \newtheorem{Definition}[Theorem]{Definition}
 \newtheorem{Assumption}[Theorem]{Assumption}
 \theoremstyle{remark}
 \newtheorem{Remark}[Theorem]{Remark}

\renewcommand*{\mid}{:\,}
\renewcommand{\d}{\mathrm{d}}

\renewcommand{\i}{\mathrm{i}}

\DeclareMathOperator{\dom}{dom}

\renewcommand{\i}{\mathrm{i}}
\renewcommand*{\mid}{:\,}
\renewcommand{\d}{\mathrm{d}}

\newcommand{\norm}[1]{\| #1\|}

\makeatletter
\newsavebox{\wonderh@t}
\newlength{\wi@th}
\newlength{\oldwi@th}
\newlength{\hei@th}
\def\hatjay{%
  \sbox{\wonderh@t}{\textit{J}}
  \setlength{\wi@th}{\wd\wonderh@t}
  \setlength{\oldwi@th}{\wd\wonderh@t}
  \setlength{\hei@th}{\ht\wonderh@t}
  \addtolength{\hei@th}{1pt}
  \addtolength{\wi@th}{1pt}
  \setlength{\wd\wonderh@t}{\wi@th}
  \setlength{\ht\wonderh@t}{\hei@th}
  \makebox[\oldwi@th][l]{\hbox{$\hat{\usebox{\wonderh@t}}$}}
}
\makeatother

\begin{document}
%%%%%%%%%%%%%%%%%%%%%%%%%%%%%%%%%%%%%%%%%%

\title[Time-Dependent Hamiltonians with a Constant Form Domain]{}

	\begin{center}
	\LARGE\textsc{
	On the Schrödinger Equation for Time-Dependent Hamiltonians with a Constant Form Domain }\par \bigskip
	
	\normalsize
	\authorfootnotes
	Aitor Balmaseda\footnote{abalmase@math.uc3m.es}\textsuperscript{,1},  
	Davide Lonigro\footnote{davide.lonigro@ba.infn.it}\textsuperscript{,2,3}, and
	Juan Manuel Pérez-Pardo\footnote{jmppardo@math.uc3m.es}\textsuperscript{,1,4} \par \bigskip
	
	\textsuperscript{1}\footnotesize Departamento de Matemáticas, Universidad Carlos III de Madrid, Avda. de la Universidad 30, 28911  Madrid, Spain \par
	\textsuperscript{2}\footnotesize Dipartimento di Fisica and MECENAS, Università di Bari, I-70126 Bari, Italy \par
	\textsuperscript{3}\footnotesize Istituto Nazionale di Fisica Nucleare, Sezione di Bari, I-70126 Bari, Italy \par
	\textsuperscript{4}\footnotesize Instituto de Ciencias Matemáticas (CSIC - UAM - UC3M - UCM) ICMAT, C/ Nicolás Cabrera 13--15, 28049 Madrid, Spain.
	\par \bigskip
	
\end{center}

\date{}

\begin{abstract}
We study two seminal approaches, developed by B.~Simon and J.~Kisyński, to the well-posedness of the Schrödinger equation with a time-dependent Hamiltonian. In both cases the Hamiltonian is assumed to be semibounded from below and to have constant form domain but a possibly non-constant operator domain. The problem is addressed in the abstract setting, without assuming any specific functional expression for the Hamiltonian. The connection between the two approaches is the relation between sesquilinear forms and the bounded linear operators representing them. We provide a characterization of continuity and differentiability properties of form-valued and operator-valued functions which enables an extensive comparison between the two approaches and their technical assumptions.
\end{abstract}

\maketitle
\vspace{-1cm}

\numberwithin{equation}{section}

%%%%%%%%%%%%%%%%%%%%%%%%%%%%%%%%%%%%%%%%%%%%%%%%%%%%%%%%%%%%%%%
%%%%%%%%%%%%%%%%%%%%%%%%%%%%%%%%%%%%%%%%%%%%%%%%%%%%%%%%%%%%%%%

\section{Introduction}
	A quantum dynamical system is a first-order linear evolution equation on a separable, complex Hilbert space $\mathcal{H}$, where the evolution is determined by a family of densely defined, self-adjoint operators $\{H(t)\}_{t\in I}$, where $I\subset\mathbb{R}$ is an interval. This family of operators is called the Hamiltonian of the system. To~simplify the notation, it is denoted by $H(t)$. In~the most general case, this represents the dynamics of a quantum system subjected to external time-varying forces. The~dynamics is given by the time-dependent (or non-autonomous) Schrödinger equation:

	\begin{equation*}
		\frac{\mathrm{d}}{\mathrm{d}t}\Psi(t)=-\i H(t)\Psi(t),
	\end{equation*}
	with some initial datum $\Psi(s)=\Psi_0$.

	Different from what happens in the autonomous case (i.e., $H(t)=H$), where the existence and uniqueness of the solution are guaranteed by the celebrated Stone theorem on one-parameter unitary groups~\cite{Stone1930,Stone1932,vonNeumann1932}, additional regularity conditions are required when $H(t)$ depends nontrivially on time. In~particular, when dealing with unbounded Hamiltonians, the~domain $\dom H(t)$ can possibly depend nontrivially on $t$. In~such cases, even by requiring the expression of the Hamiltonian to be reasonably ``well-behaved'' (e.g., continuously differentiable) as a function of time, the~well-posedness of the problem above is not ensured; standard methods such as product integrals or Dyson expansions cannot be applied~directly.

Singular potentials in the Schrödinger equation are a convenient way of analysing different types of nontrivial boundary conditions. One example is point-like interactions, which model impurities. This has applications in physically relevant systems such as propagation in quantum waveguides~\cite{facchi2021spectral, leonforte2021dressed,Lonigro_2021} or quantum dots~\cite{dell2005singular}. Furthermore, recent results allowing for the numerical computation of the spectrum for nontrivial boundary conditions~\cite{ibort2013numerical, lopez2017finite} opened the possibility to control the state of a quantum system by modifying the boundary conditions. The~latter has implications in the development of new quantum~technologies.

 The problem of time-dependent Hamiltonians for the class of models $H(t)=-\Delta+V(t)$, with~$\Delta$ being the Laplace operator on $L^2(\Omega)$, $\Omega\subset\mathbb{R}^n$, and~$V(t)$ being a time-dependent potential, has been studied extensively in the literature. However, these results cannot be applied directly to treat the cases with non-constant boundary conditions, e.g.,~singular potentials, and~one needs to rely on more abstract settings or take a case-by-case approach. The~situation where $V(t)$ is a time-dependent singular potential has also been considered~\cite{Wuller1986,Yajima1987,Yajima2011,Hughes1995,Carles2011}, for~example point (also called Dirac) interactions with a time-dependent position or strength~\cite{Mantile2011,Posilicano2007,DellAntonioFigariTeta2000}, possibly satisfying specific scaling properties~\cite{Samaj2002,Combescure1988}. In~these latter situations, the domains of the unbounded Hamiltonians $H(t)$ depend explicitly on the parameter $t$. Alternatively, one may consider the Laplace operator on a bounded region with time-dependent boundary conditions. In~such a case, the~time dependence of the operator is entirely encoded in its domain. Results in this direction have been found for a quantum particle in a one-dimensional cavity with moving walls, which can be transformed into a fixed-boundary problem~\cite{MunierBurganFeixEtAl1981,MakowskiDembinski1991,DiMartinoAnzaFacchiEtAl2013,FacchiGarneroMarmoEtAl2016}. More recently, time-dependent boundary conditions in a graph-like manifold, potentially relevant in quantum control theory, have been investigated~\cite{BalmasedaPerezPardo2019,BalmasedaLonigroPerezPardo2021}.
 
 From a mathematical point of view, the non-autonomous Schrödinger equation is a first-order linear evolution problem in a Hilbert space with an unbounded generator that depends on time, including the case in which its domain depends on time. The~first one addressing the problem in this abstract setting was T.~Kato~\cite{Kato1953,Kato1956}.
	In this first approach, he considered the operator domain $\dom H(t)$ to be constant and provided sufficient conditions for the existence of solutions.
	Furthermore, J.L.~Lions~\cite{Lions1961} addressed the existence of solutions for problems of this type for the case in which the operator defining the dynamics, i.e.,\ the Hamiltonian, is an elliptic differential operator with smooth coefficients.
	K.~Yosida~\cite{Yosida1963, Yosida1965} studied the case in which $H(t)$ is positive and~introduced a family of approximations $H_n(t) = H(t) (n^{-1} - H(t))^{-1}$, currently commonly known as Yosida's approximations. Even if his intention was to address the problem on which $\dom H(t)$ was not constant, T.~Kato noticed that Yosida's original assumptions implied that $\dom H(t)$ had to be constant; see the remark on page 429 in~\cite{Yosida1965}.

	It was J.~Kisyński~\cite{Kisynski1964} to be the first one to improve Yosida's techniques in a way allowing for time-varying domains.
	Independently, B.~Simon~\cite{Simon1971} took an approach very similar to that of J.~Kisyński, albeit only considering the case in which the form domain of the operator is constant; see Definition~\ref{def:hamiltonian_const_form}. In~the case of strictly positive operators, the form domain is the domain of the operator $H(t)^{1/2}$ obtained by the functional calculus of self-adjoint operators. 
	There is also another work by T.~Kato~\cite{Kato1973} addressing the problem for time-dependent domains, and~further refinements of this approach have been developed over the years.
	A more recent work on the abstract non-autonomous Schrödinger problem is~\cite{MasperoRobert2017}.
	The stability properties of time-dependent Hamiltonians have also been investigated via non-standard analysis techniques~\cite{Sloan1981} and, more recently, in~\cite{BalmasedaLonigroPerezPardo2021}, where sharper bounds to the norms of the solutions obtained by B.~Simon were found.
Finally, let us point out that many practical models use the nonlinear Schrödinger equation, i.e.,~they involve nonlinear terms, and a~certain energy law is preserved. For~instance, there are some existing works on the existence and uniqueness of solutions for certain nonlinear hyperbolic PDEs (with no dissipation) \cite{wang2010global}, as~well as the numerical approximation to these PDEs, such as~\cite{wang2011energy, baskaran2013energy, baskaran2013convergence, wang2015energy}.

The case of study of the research presented in this article is when $H(t)$ is a family of self-adjoint operators, uniformly bounded from below and with a constant form domain. No other assumptions on the particular form of the operator were made. The~problem of the existence and uniqueness of the solutions was studied by B.~Simon and J.~Kisyński with similar approaches. B.~Simon pointed out the relation between both approaches in his book (see footnote 21 in \cite{Simon1971}, p. 56, and his remark after Theorem~II.27). Referring to his approach to the existence theorem (Theorem~\ref{thm:simon}), he wrote (\cite{Simon1971}, pp. 58--59):
 \begin{quote}
 ``[T]his theorem (in a slightly different form) is contained in the work of Kisyński~\cite{Kisynski1964}.
 [...]
 Yosida's techniques are also behind Kisyński's approach.
 [...]
 I discovered Kisyński's paper only after completing this proof myself and have not checked carefully the differences, if~any, in~the details of the two proofs.''
 \end{quote}
 The connection between both approaches is not immediate, and even the relation between the necessary conditions taken is not completely obvious.
 To the best of our knowledge, the~details of this connection are yet to be thoroughly~understood.

	The aim of this work is to revise, in~a common language, both approaches to the well-posedness problem for the non-autonomous Schrödinger equation with a constant form domain and (possibly) non-constant operator domain and~to clarify the connections between them.
 The connection between J.~Kisyński's and B.~Simon's is subtle, and~the relationship between their approaches and the necessary conditions taken is not~straightforward.
 
	In showing the explicit relation between both approaches, we revise the state-of-the-art of the problem. Since we did not assume any specific expression for the Hamiltonian $H(t)$, our work may also provide a useful reference for mathematical and theoretical physicists interested in such problems.
	This may inspire, for~example, new developments in the study of quantum Hamiltonians with time-dependent boundary conditions. We applied our results to a simple model with a point-like interaction whose strength depends on time. The~example shows explicitly the conditions imposed by B. Simon and J. Kisyński and highlights the analogies and subtle differences between the~approaches.

The paper is organised as follows. In~Section~\ref{sec:preliminaries}, we study the regularity properties of form-valued and operator-valued functions (Section~\ref{sec:op_form_functions}). To~do this, we recall first some basic notions about operators and forms on a Hilbert space (Section~\ref{sec:hilbert_operators}) and we present the canonical construction of the scale of Hilbert spaces associated with a semibounded self-adjoint operator (Section~\ref{sec:hilbert_scales}). In~Section~\ref{sec:existence_dynamics}, we introduce the non-autonomous Schrödinger equation, both in its strong and weak formulation, for~a Hamiltonian with a constant form domain (Section~\ref{subsec:quantum_dynamics}). Then, we revise B.~Simon's (Section~\ref{subsec:simon}) and J.~Kisyński's (Section~\ref{subsec:kisynski}) approaches to the well-posedness problem for the Schrödinger equation. Finally, in \mbox{Section~\ref{subsec:simon_vs_kisynski}}, we analyse the relations between both approaches using the results obtained in Section~\ref{sec:preliminaries}.

	\section{Regularity of Operator-Valued and Form-Valued~Functions}\label{sec:preliminaries}
	\subsection{Operators and Forms on Hilbert~Spaces} \label{sec:hilbert_operators}
	Let $\mathcal{H}$ be a complex, separable Hilbert space with inner product $\langle \cdot, \cdot \rangle$ and norm $\|\cdot\|$. 
	In this work, we only considered operators $T: \dom T \subset \mathcal{H} \to \mathcal{H}$ densely defined on $\mathcal{H}$.
	Self-adjoint operators semibounded from below are closely related to closed sesquilinear forms; see for instance \cite{Kato1995}, Chapter~VI, and \cite{ReedSimon1980}, Section~VIII.6.
	Throughout this work, we observe the convention establishing that a sesquilinear form is anti-linear on its first argument and linear on the second one.
	We only considered Hermitian sesquilinear forms $h: \dom h \times \dom h \to \mathbb{C}$ defined on dense subspaces of the Hilbert space $\mathcal{H}$: i.e.,~$\dom h$ is a dense subspace of $\mathcal{H}$.%

	For convenience, we state the following consequence of the Uniform Boundedness Principle.
	\begin{Theorem} \label{thm:unif_boundedness_forms}
		Let $\mathcal{H}$ be a Hilbert space, and~let $\mathcal{F}$ be a family of bounded sesquilinear forms on $\mathcal{H}$.
		If for every $\Psi, \Phi \in \mathcal{H}$, the set $\{|h(\Psi, \Phi)| \mid h \in \mathcal{F}\}$ is bounded, then the set:
		\begin{equation*}
			\bigl\{|h(\Psi, \Phi)| \mid \Psi,\Phi \in \mathcal{H},\,\|\Psi\| \leq 1, \|\Phi\| \leq 1, \, h \in \mathcal{F}\bigr\}
		\end{equation*}
		is bounded.
	\end{Theorem}
	\begin{proof}
		For a fixed $\Psi \in \mathcal{H}$, consider the family of bounded operators from $\mathcal{H}$ to $\mathbb{C}$:
		\begin{equation*}
			\mathcal{T}_\Psi = \{\Phi \mapsto h(\Psi,\Phi) \mid h \in \mathcal{F}\}.
		\end{equation*}
		By
 assumption, for~every $\Psi, \Phi \in \mathcal{H}$, there is a constant $K_{\Psi,\Phi} > 0$ such that $|h(\Psi, \Phi)| \leq K_{\Psi,\Phi}$.
		Therefore, by~the Uniform Boundedness Principle, the~set:
		\begin{equation*}
			% \{\|T\| \mid T \in \mathcal{T}_\Psi\} =
			\{|h(\Psi, \Phi)| \mid \Phi \in \mathcal{H},\,\|\Phi\| \leq 1, \, h \in \mathcal{F}\}
		\end{equation*}
		is bounded.
		Applying again the Uniform Boundedness Principle to the family of bounded operators $\{\Psi \mapsto \overline{h(\Psi,\Phi)} \mid \Phi \in \mathcal{H},\,\|\Phi\| \leq 1, \, h \in \mathcal{F}\}$, the~result follows.
	\end{proof}

	An important concept to relate sesquilinear forms with operators is the notion of closed and semibounded quadratic form. Recall that a Hermitian sesquilinear form is said to be \textit{semibounded} if there exists $m>0$ such that $h(\Phi,\Phi) > -m \norm{\Phi}^2$ for all $\Phi\in\dom h$; in such a case, $m$ is its~semibound.

	\begin{Definition}\label{def:graph-norm}
		Let $h$ be a semibounded Hermitian sesquilinear form with dense domain $\dom h$, and~let $m$ be the semibound of $h$.
		We define the {graph norm} of the sesquilinear form $h$ by:
		\begin{equation*}
			\|\Phi\|_h:=\sqrt{(1+m)\|\Phi\|^2+h(\Phi,\Phi)},\quad \Phi\in\dom h.
		\end{equation*}
		We say that $h$ is {closed} if $\dom h$ is closed with respect to the graph norm $\|\cdot\|_h$.
	\end{Definition}

	We recall next an important result; cf. \cite{Kato1995}, Section~VI.2.
	\begin{Theorem}[Representation theorem]\label{thm:repKato}
		Let $h$ be a Hermitian, closed, semibounded sesquilinear form with dense domain $\dom h \subset \mathcal{H}$.
		Then, there exists a unique, self-adjoint, semibounded operator $T$ with domain $\mathcal{D}$ and the same lower bound, such~that:
		\begin{enumerate}[label=\textit{(\roman*)},nosep, leftmargin=*]
			\item $\Phi \in \mathcal{D}$ if and only if $\Phi \in \dom h$ and there exists
			$\chi \in \mathcal{H}$ such that:
			\begin{equation*}
				h(\Psi, \Phi) = \langle \Psi, \chi \rangle, \qquad \forall \Psi \in \dom h;
			\end{equation*}
			\item $h(\Psi, \Phi) = \langle \Psi, T\Phi \rangle$ for any $\Psi \in \dom h$,
			$\Phi \in \mathcal{D}$;
			\item $\mathcal{D}$ is a core for $h$, that is
			$\overline{\mathcal{D}}^{\|\cdot\|_h} = \dom h$.
		\end{enumerate}
	\end{Theorem}
	Note that this theorem establishes a one-to-one correspondence between closed, semibounded Hermitian sesquilinear forms and semibounded self-adjoint operators and~motivates the following definition.
	\begin{Definition} \label{def:representing_op}
		Let $h$ be a closed, semibounded, Hermitian sesquilinear form.
		The operator $T$ given in Theorem~\ref{thm:repKato} is said to be the operator {representing} $h$.
		Conversely, $h$ is called the sesquilinear form {represented by} $T$.
	\end{Definition}

	\subsection{Scales of Hilbert~Spaces} \label{sec:hilbert_scales}
	The notion of scales of Hilbert spaces (also known as \emph{Gelfand triples}) plays a central role in this article. We review the basic ideas of the construction in this subsection.
	Details and proofs of the following statements can be found for instance in \cite{Berezanskii1968}, Ch. I.

	Let $\mathcal{H}^+ \subset \mathcal{H}$ be a dense subspace of the Hilbert space $\mathcal{H}$, and~let $\langle \cdot, \cdot \rangle_+$ be an inner product endowing $\mathcal{H}^+$ with the structure of a Hilbert space and such that the associated norm, $\|\cdot\|_+$, satisfies:
	\begin{equation*}
		\|\Phi\| \leq \|\Phi\|_+, \qquad \Phi \in \mathcal{H}^+.
	\end{equation*}
	By the Riesz representation theorem, the~restriction of the inner product of $\mathcal{H}$ on $\mathcal{H}^+$ can be represented using the inner product in $\mathcal{H}^+$, i.e.,~there exists an operator $\hatjay: \mathcal{H} \to \mathcal{H}^+$ such that:
	\begin{equation*}
		\langle \Psi, \Phi \rangle = \langle \hatjay \Psi, \Phi \rangle_+, \qquad
		\Psi \in \mathcal{H}, \Phi \in \mathcal{H}^+.
	\end{equation*}
	This operator is injective and~allows defining another inner product on $\mathcal{H}$,
	\begin{equation*}
		\langle \cdot, \cdot \rangle_- \coloneqq \langle \hatjay\cdot, \hatjay\cdot \rangle_+.
	\end{equation*}
	Let $\mathcal{H}^-$ be the completion of $\mathcal{H}$ with respect to the norm $\|\cdot\|_-$ associated with $\langle \cdot, \cdot \rangle_-$.
	The operator $\hatjay$ can be extended by continuity to an isometric bijection $J: \mathcal{H}^- \to \mathcal{H}^+$.
	The spaces $\mathcal{H}, \mathcal{H}^\pm$ form the scale of Hilbert spaces $\mathcal{H}^+ \subset \mathcal{H} \subset \mathcal{H}^-$.

	Finally, since:
	\begin{equation*}
		|\langle \Psi, \Phi \rangle| = |\langle J\Psi, \Phi \rangle_+| \leq \|J\Psi\|_+ \|\Phi\|_+
		= \|\Psi\|_- \|\Phi\|_+, \quad
		\Psi \in \mathcal{H}, \Phi \in \mathcal{H}^+,
	\end{equation*}
	the inner product on $\mathcal{H}$ can be continuously extended to a pairing:
	\begin{equation*}
		(\cdot, \cdot): \mathcal{H}^- \times \mathcal{H}^+ \cup \mathcal{H}^+ \times \mathcal{H}^- \to \mathbb{C}.
	\end{equation*}
	Note also that, by~definition,
	\begin{equation*}
		\langle \Psi, \Phi \rangle_{\pm} = (\Psi, \hatjay^{\mp1} \Phi), \qquad \Psi,\Phi \in \mathcal{H}^\pm.
	\end{equation*}

	Let us introduce now the scale of Hilbert spaces associated with a sesquilinear form.
	Let $h: \mathcal{H}^+ \times \mathcal{H}^+ \to \mathbb{C}$ be a Hermitian, strictly positive sesquilinear form such that $\mathcal{H}^+$ is complete with respect to the norm induced by the inner product $\langle \cdot, \cdot \rangle_+ \coloneqq h(\cdot, \cdot)$ and satisfying:
	\begin{equation*}
		\|\Phi\|^2 \leq h(\Phi, \Phi), \qquad \Phi \in \mathcal{H}^+.
	\end{equation*}
	The construction above can therefore be applied to define an associated scale of Hilbert spaces $\mathcal{H}^+ \subset \mathcal{H} \subset \mathcal{H}^-$.

	Let $H$ be the positive, self-adjoint operator representing $h$, that is,
	\begin{equation*}
		\langle \Psi, \Phi \rangle_+ = h(\Psi, \Phi) = \langle \Psi, H\Phi \rangle, \quad
	\end{equation*}
	for all $\Psi \in \mathcal{H}^+$ and $\Phi \in \dom H$.
	Note that, if~$H$ is strictly positive, $H^{-1} \in \mathcal{B}(\mathcal{H})$ is well defined and $\hatjay = H^{-1}$.
	Therefore, the~operators $H, H^{-1}$ can be extended to $\tilde{H} = J^{-1} \in \mathcal{B}(\mathcal{H}^+, \mathcal{H}^-)$ and $\tilde{H}^{-1} = J \in \mathcal{B}(\mathcal{H}^-, \mathcal{H}^+)$.
	To simplify the notation, we denote these extensions and the original operators by the same symbols, $H$ and $H^{-1}$.
	Note also that, since $H$ is self-adjoint and positive, then $H^{\pm\sfrac{1}{2}}$ are well-defined and the following identities hold:
	\begin{equation*}
		H^{\pm\sfrac{1}{2}} \mathcal{H}^\pm = \mathcal{H}, \qquad H^{\pm\sfrac{1}{2}}\mathcal{H} = \mathcal{H}^{\mp}
	\end{equation*}
and
	\begin{equation*}
		\langle \Psi, \Phi \rangle_{\pm} = \langle H^{\pm\sfrac{1}{2}}\Psi, H^{\pm\sfrac{1}{2}}\Phi \rangle,
	\end{equation*}
	for $\Phi,\Psi$ in the appropriate~spaces. 

	This discussion leads to the following definition.
	\begin{Definition} \label{def:hilbert-scales-h}
		Let $\mathcal{H}$ be a Hilbert space with associated inner product $\langle \cdot, \cdot \rangle$, $\mathcal{H}^+ \subset \mathcal{H}$ a dense subspace, and~let $h: \mathcal{H}^+ \times \mathcal{H}^+ \to \mathbb{C}$ be a Hermitian, strictly positive, closed sesquilinear form.
		Denote by $H$ the strictly positive self-adjoint operator representing $h$ (cf.\ Definition~\ref{def:representing_op}), and~define the inner products:
		\begin{equation*}
			\langle \Psi, \Phi \rangle_\pm = (\Psi, H^{\pm 1}\Phi) \qquad \Psi, \Phi \in \mathcal{H}^{\pm},
		\end{equation*}
where $\mathcal{H}^-$ is the closure of $\mathcal{H}$ with respect to the norm $\|\cdot\|_-$ induced by $\langle \cdot, \cdot \rangle_-$.
		The {{scale of Hilbert spaces associated with} $h$} %MDPI: is the normal necessary in Definition? please check all normal format in Definition, Lamma, Theorem, ....
 is the scale:
		\begin{equation*}
			(\mathcal{H}^+, \langle \cdot, \cdot \rangle_+)
			\subset (\mathcal{H}, \langle \cdot, \cdot \rangle)
			\subset (\mathcal{H}^-, \langle \cdot, \cdot \rangle_-).
		\end{equation*}
		The {{scale of Hilbert spaces associated with a semibounded sesquilinear form}} $\tilde{h}$ with semibound $m$ is the scale associated with the strictly positive sesquilinear form $h(\Psi, \Phi) = \tilde{h}(\Psi, \Phi) + m + 1$.
	\end{Definition}
	The scale of Hilbert spaces associated with a semibounded self-adjoint operator will be the scale of Hilbert spaces associated with the unique quadratic form representing~it.

	Let us end this subsection with a useful result on families of scales of Hilbert spaces sharing a common domain $\mathcal{H}^+$.
	\begin{Theorem} \label{thm:eqiv+_equiv-}
		Let $\mathcal{H}^+$ be a dense subset of a Hilbert space $\mathcal{H}$, and let $\langle \cdot, \cdot \rangle_{+,1}$ and $\langle \cdot, \cdot \rangle_{+,2}$ be inner products that give rise to the scales of Hilbert spaces:
		\begin{equation*}
			(\mathcal{H}^+, \langle \cdot, \cdot \rangle_{+,i}) \subset \mathcal{H} \subset (\mathcal{H}^-, \langle \cdot, \cdot \rangle_{-,i}),\quad i=1,2.
		\end{equation*}
		Denote by $\|\cdot\|_{\pm,i}$ the norm on $\mathcal{H}^\pm$, and~let $c > 0$.
		The following statements are~equivalent:
		\begin{enumerate}[label=\textit{(\roman*)},nosep,leftmargin=*]
			\item \label{item:equiv+}
			For all $\Phi \in \mathcal{H}^+$, $c^{-1} \|\Phi\|_{+,1} \leq \|\Phi\|_{+,2} \leq c \|\Phi\|_{+,1}$;
			\item \label{item:equiv-}
			For all $\Phi \in \mathcal{H}^-$, $c^{-1} \|\Phi\|_{-,1} \leq \|\Phi\|_{-,2} \leq c \|\Phi\|_{-,1}$.
		\end{enumerate}
	\end{Theorem}
	\begin{proof}
		Let $A_i$ be the strictly positive self-adjoint operator with domain $\mathcal{H}^+$ such that:
		\begin{equation*}
			\langle \Psi, \Phi \rangle_{+,i} = \langle A_i \Psi, A_i \Phi \rangle,\quad
			i = 1, 2.
		\end{equation*}
		By the closed graph theorem, the~operator defined by $T \coloneqq A_1 A_2^{-1}$ is a bounded operator on $\mathcal{H}$; therefore, given $\Phi\in\mathcal{H}^+$,
		\begin{equation*}
			\|\Phi\|_{+,1} = \|A_1A_2^{-1}A_2\Phi\|
			\leq \|T\| \|A_2\Phi\| = \|T\| \|\Phi\|_{+,2}.
		\end{equation*}
		Analogously, one can obtain $\|\Psi\|_{+,2} \leq \|T^{-1}\| \|\Psi\|_{+,1}$ for $\Psi \in \mathcal{H}^+$.

		Using the adjoint of $T$, $T^\dagger = A_2^{-1} A_1$, one can prove similar inequalities for the norms $\|\cdot\|_{-,i}$, $i=1,2$:
\begin{equation} \label{eq:Tdagger_minus_norm}
			\|\Phi\|_{-,1} \leq \|T^{-\dagger}\| \|\Phi\|_{-,2}, \quad
			\|\Phi\|_{-,2} \leq \|T^\dagger\| \|\Phi\|_{-,1}.
		\end{equation}

		If \ref{item:equiv+} holds, one has:
		\begin{equation*}
			\|T\Phi\| = \|A_2^{-1} \Phi\|_{+,1} \leq c \|A_2^{-1}\Phi\|_{+,2} = c \|\Phi\|.
		\end{equation*}
		Similarly, it follows that $\|T^{-1}\| \leq c$, and~therefore, $\|T^{\pm \dagger}\| = \|T^{\pm 1}\| \leq c$.
		By Equation~\eqref{eq:Tdagger_minus_norm}, \ref{item:equiv-} follows.

		The other implication, i.e.,\ \ref{item:equiv-} implies \ref{item:equiv+}, is proven analogously.
	\end{proof}

	\subsection{Regularity of Operator-Valued and Form-Valued~Functions} \label{sec:op_form_functions}{\strut}
	In the statement of the following results, $\mathcal{H}^+$ is a Hilbert space with norm $\norm{\cdot}_+$ and inner product $\langle{\cdot},{\cdot}\rangle_+$, not necessarily part of a scale of Hilbert spaces. We use this notation to ease the reading, as later we use these results in the case that $\mathcal{H}^+$ is part of a scale. Let $I \subset \mathbb{R}$ be a real interval, and let $\mathcal{V} = \{v_t\}_{t \in I}$ be a family of bounded sesquilinear forms on $\mathcal{H}^+$.
	The aim of this subsection is to investigate the relationship between the regularity of the functions $t \in I \mapsto v_t(\Psi,\Phi) \in \mathbb{C}$ for $\Psi, \Phi \in \mathcal{H}^+$ fixed and the regularity of the form-valued functions $t \in I \mapsto v_t \in \mathcal{V}$.
	Let us first introduce some technical lemmas.
	\begin{Lemma} \label{lemma:sup-cont}
		Let $\mathcal{F}$ be an equicontinuous family of functions from a Hilbert space $\mathcal{H}^+$ to $\mathbb{C}$.
		Then, the function $f: \Phi \in \mathcal{H}^+ \mapsto \sup_{F \in \mathcal{F}} |F(\Phi)| \in \mathbb{C}$ is continuous.
	\end{Lemma}
	\begin{proof}
		Let $\Phi, \Phi_0 \in \mathcal{H}^+$.
		It holds that:
		\begin{alignat*}{2}
			\sup_{F \in \mathcal{F}} |F(\Phi)| &= \sup_{F \in \mathcal{F}} [|F(\Phi)| - |F(\Phi_0)| + |F(\Phi_0)|]\\
			&\leq \sup_{F \in \mathcal{F}} [|F(\Phi)| - |F(\Phi_0)|] + \sup_{F \in \mathcal{F}} |F(\Phi_0)|,
		\end{alignat*}
		Combining this inequality with the one obtained interchanging the roles of $\Phi$ and $\Phi_0$, it follows that:
		\begin{equation*}
			\left|\sup_{F \in \mathcal{F}} |F(\Phi)| - \sup_{F \in \mathcal{F}} |F(\Phi_0)|\right|
			\leq \sup_{F \in \mathcal{F}} \left||F(\Phi)| - |\vphantom{\lim_2}F(\Phi_0)|\right|.
		\end{equation*}

		Therefore, one has:
		\begin{equation*}
			|f(\Phi) - f(\Phi_0)| = \left|\sup_{F \in \mathcal{F}} |F(\Phi)| - \sup_{F \in \mathcal{F}} |F(\Phi_0)|\right|
			\leq \sup_{F \in \mathcal{F}} \left|\vphantom{\lim_2}|F(\Phi)| - |F(\Phi_0)|\right|.
		\end{equation*}
		From this inequality and the equicontinuity of $\mathcal{F}$, the result follows.
	\end{proof}

	\begin{Lemma} \label{lemma:locally_unif_bound}
		For every $t \in I \subset \mathbb{R}$, let $v_t: \mathcal{H}^+ \times \mathcal{H}^+\to\mathbb{C}$ be a bounded sesquilinear form.
		Let $t_0 \in I$, and suppose that $\lim_{t \to t_0} v_t(\Psi, \Phi)$ exists for every $\Psi, \Phi \in \mathcal{H}^+$.
		Then, there is a neighbourhood $B_{t_0}$ of $t_0$ and a constant $K$ such that:
		\begin{equation*}
			|v_t(\Psi,\Phi)| \leq K \|\Psi\|_+ \|\Phi\|_+, \qquad
			\forall \Psi,\Phi \in \mathcal{H}^+, \quad \forall t \in B_{t_0}.
		\end{equation*}
		Moreover, $L(\Psi, \Phi) = \lim_{t \to t_0} v_t(\Psi, \Phi)$ defines a bounded sesquilinear form on $\mathcal{H}^+$ with:
		\begin{equation*}
			|L(\Psi,\Phi)| \leq K \|\Psi\|_+ \|\Phi\|_+.
		\end{equation*}
	\end{Lemma}
	\begin{proof}
		For every $\Psi, \Phi \in \mathcal{H}^+$, the~existence of $\lim_{t \to t_0} v_t(\Psi, \Phi)$ implies that there is a neighbourhood $B_{t_0}$ of $t_0$ and a constant $K_{\Psi,\Phi} > 0$ such that, for~$t \in B_{t_0}$, we have $|v_t(\Psi, \Phi)| \leq K_{\Psi,\Phi}$.
		Therefore, by~Theorem~\ref{thm:unif_boundedness_forms}, there is $K > 0$ such that, for~every $t \in B_{t_0}$ and every $\Psi, \Phi \in \mathcal{H}^+$,
		\begin{equation*}
			|v_t(\Psi,\Phi)| \leq K \|\Psi\|_+ \|\Phi\|_+.
		\end{equation*}
		It is straightforward to check that $L(\Psi, \Phi)$ is a sesquilinear form, and~since the previous bound holds independently of $t \in B_{t_0}$, it follows that $|L(\Psi,\Phi)| \leq K \|\Psi\|_+ \|\Phi\|_+$.
	\end{proof}

	\begin{Lemma} \label{lemma:form_uniform_limit}
		For each $t \in I \subset \mathbb{R}$, let $v_t: \mathcal{H}^+ \times \mathcal{H}^+$ be a bounded sesquilinear form such that $\lim_{t \to t_0} v_t(\Psi, \Phi)$ exists for every $\Psi, \Phi \in \mathcal{H}^+$, and denote by $L$ the bounded sesquilinear form defined by $L(\Psi, \Phi) = \lim_{t \to t_0} v_t(\Psi, \Phi)$.
		Then:
		\begin{equation*}
			\adjustlimits\lim_{t \to t_0} \sup_{\Psi,\Phi \in \mathcal{H}^+ \setminus \{0\}} \frac{|v_t(\Psi, \Phi) - L(\Psi,\Phi)|}{\|\Phi\|_+ \|\Psi\|_+} = 0.
		\end{equation*}
	\end{Lemma}
	\begin{proof}
		Fix $\Psi_0, \Phi_0 \in \mathcal{H}^+$ such that $\|\Psi_0\|_+ \leq 1$ and $\|\Phi_0\|_+ \leq 1$, and~denote by $B_{t_0}$ the neighbourhood of Lemma \ref{lemma:locally_unif_bound}.
		By the definition of the limit, for~any $\varepsilon > 0$, there exists $\delta$ such that, defining $B_{t_0}(\delta) \coloneqq \{t \in I \mid |t - t_0| < \delta\} \subset B_{t_0}$, for~$t \in B_{t_0}(\delta)$:
		\begin{equation*}
			|v_t(\Psi_0, \Phi_0) - L(\Psi_0, \Phi_0)| < \frac{\varepsilon}{2}.
		\end{equation*}
		Therefore, for~$t \in B_{t_0}(\delta)$,
\begin{equation} \label{eq:form_unif_lim_1}
			|v_t(\Psi_0, \Phi) - L(\Psi_0, \Phi)| \leq
			\sup_{t \in B_{t_0}(\delta)} |v_t(\Psi_0,\Phi-\Phi_0) - L(\Psi_0,\Phi-\Phi_0)|
			+ \frac{\varepsilon}{2}.
		\end{equation}

		Consider the function $f_{\Phi_0}: \mathcal{H}^+ \to [0, \infty)$ defined by:
		\begin{equation*}
			f_{\Phi_0}(\Phi) \coloneqq \sup_{t \in B_{t_0}(\delta)} |v_t(\Psi_0,\Phi-\Phi_0) - L(\Psi_0,\Phi-\Phi_0)|.
		\end{equation*}
		By Lemma \ref{lemma:locally_unif_bound}, the~family of functions $\{\Phi \mapsto v_t(\Psi_0,\Phi-\Phi_0) - L(\Psi_0,\Phi-\Phi_0) \mid t \in B_{t_0}(\delta)\}$ is equicontinuous, and~therefore, Lemma \ref{lemma:sup-cont} implies that $f_{\Phi_0}$ is a continuous map.
		This implies that $f_{\Phi_0}$ is also weakly continuous, and~therefore, $U_{\Phi_0} \coloneqq f_{\Phi_0}^{-1}(\{x\in \mathbb{R}: |x| < {\varepsilon}/{2}\})$ is an open neighbourhood of $\Phi_0$ in the weak topology of $\mathcal{H}^+$.
		By Eq.~\eqref{eq:form_unif_lim_1} and the definition of $U_{\Phi_0}$, for~every $\Phi \in U_{\Phi_0}$ and every $t \in B_{t_0}(\delta)$,
		\begin{equation*}
			|v_t(\Psi_0, \Phi) - L(\Psi_0, \Phi)| \leq \varepsilon.
		\end{equation*}
		That is, for~every $\Psi_0,\Phi_0 \in \mathcal{H}^+$, $\|\Phi_0\|_+ \leq 1$, and $\varepsilon > 0$, there is $\delta$ and a neighbourhood of $\Phi_0$ in the weak topology, $U_{\Phi_0}$, such that, for~every $\Phi \in U_{\Phi_0}$ and $|t - t_0| < \delta$,
		\begin{equation*}
			|v_t(\Psi_0, \Phi) - L(\Psi_0, \Phi)| < \varepsilon.
		\end{equation*}
		The family $\{U_{\Phi}: \|\Phi\|_+ \leq 1,\,\Phi \in \mathcal{H}^+\}$ is an open covering of the closed unit ball on $\mathcal{H}^+$, and by~the weak compacity of the unit ball, there is a finite subcovering $\{U_n\}_{n = 1}^N$.
		It follows that, for~every $U_n$, $1 \leq n \leq N$, there is $\delta_n$ such that, for~$\Phi \in U_n$ and $|t - t_0| < \delta_n$:
		\begin{equation*}
			|v_t(\Psi_0, \Phi) - L(\Psi_0, \Phi)| \leq \varepsilon.
		\end{equation*}
		Therefore, for~$|t - t_0| < \min_n \delta_n$ and every $\Phi \in \mathcal{H}^+$ with $\|\Phi\|_+ \leq 1$, one has:
		\begin{equation*}
			|v_t(\Psi_0, \Phi) - L(\Psi_0, \Phi)| \leq \varepsilon.
		\end{equation*}
		Repeating the argument for $\Psi$ yields that the limit is uniform on $\Psi, \Phi$ in the closed unit ball, from~which the result follows.
	\end{proof}

	We can now prove the main results of this subsection.
	\begin{Proposition} \label{prop:form_equicont_equidiff}
		For $t \in I \subset \mathbb{R}$, let $v_t: \mathcal{H}^+ \times \mathcal{H}^+ \to \mathbb{C}$ be a family of bounded sesquilinear forms.
		Then:
		\begin{enumerate}[label=\textit{(\roman*)},nosep,leftmargin=*,ref=\ref{prop:form_equicont_equidiff}\textit{(\roman*)}]
			\item \label{prop:equicontinuity_forms}
			If, for~every $\Psi, \Phi \in \mathcal{H}^+$, the~map $t \mapsto v_t(\Psi, \Phi)$ is continuous, then:
			\begin{equation*}
				\adjustlimits\lim_{t \to t_0} \sup_{\Psi,\Phi \in \mathcal{H}^+ \setminus \{0\}} \frac{|v_t(\Psi, \Phi) - v_{t_0}(\Psi,\Phi)|}{\|\Phi\|_+ \|\Psi\|_+} = 0.
			\end{equation*}
			If, in~addition, $I$ is compact, there exists $M>0$ such that $|v_t(\Psi,\Phi)| \leq M \|\Psi\|_+ \|\Phi\|_+$ for every $t \in I$ and every $\Psi,\Phi \in \mathcal{H}^+$;
			\item \label{prop:equidifferentiability_forms}
			If, for~every $\Psi, \Phi \in \mathcal{H}^+$, the~map $t \mapsto v_t(\Psi, \Phi)$ is differentiable, then:
			\begin{equation*}
				\adjustlimits\lim_{t \to t_0} \sup_{\Psi,\Phi \in \mathcal{H}^+ \setminus \{0\}} \frac{1}{\|\Phi\|_+ \|\Psi\|_+} \left|\frac{v_t(\Psi, \Phi) - v_{t_0}(\Psi,\Phi)}{t - t_0} - \dot{v}_{t_0}(\Psi,\Phi)\right| = 0,
			\end{equation*}
where $\dot{v}_t(\Psi,\Phi)$ denotes the derivative of $t \mapsto v_t(\Psi,\Phi)$.
			If, in~addition, $I$ is compact and $t \mapsto \dot{v}_t(\Psi,\Phi)$ is continuous for every $\Psi,\Phi$, there is $M > 0$ such that $|\dot{v}_t(\Psi,\Phi)| < M \|\Psi\|_+ \|\Phi\|_+$ for every $t \in I$.
		\end{enumerate}
	\end{Proposition}
	\begin{proof}
		By Lemma \ref{lemma:form_uniform_limit}, the~continuity of $t \mapsto v_t(\Psi,\Phi)$ for each $\Psi, \Phi \in \mathcal{H}^+$ implies:
		\begin{equation*}
			\adjustlimits\lim_{t \to t_0} \sup_{\Psi,\Phi \in \mathcal{H}^+ \setminus \{0\}} \frac{|v_t(\Psi, \Phi) - v_{t_0}(\Psi,\Phi)|}{\|\Phi\|_+ \|\Psi\|_+} = 0,
		\end{equation*}
		and its differentiability implies:
		\begin{equation*}
			\adjustlimits\lim_{t \to t_0} \sup_{\Psi,\Phi \in \mathcal{H}^+ \setminus \{0\}} \frac{1}{\|\Phi\|_+ \|\Psi\|_+} \left|\frac{v_t(\Psi, \Phi) - v_{t_0}(\Psi,\Phi)}{t - t_0} - \dot{v}_{t_0}(\Psi,\Phi)\right| = 0.
		\end{equation*}
		The uniform bound in \emph{(i)} follows from Lemma \ref{lemma:locally_unif_bound} and the compacity of $I$, while the uniform bound in \emph{(ii)} follows applying Lemma \ref{lemma:locally_unif_bound} and the continuity of $t \mapsto \dot{v}_t(\Psi, \Phi)$.
	\end{proof}

	Given a scale of Hilbert spaces $\mathcal{H}^+ \subset \mathcal{H} \subset \mathcal{H}^-$, by~the Riesz representation theorem, any bounded sesquilinear form $v: \mathcal{H}^+ \times \mathcal{H}^+ \to \mathbb{C}$ can be associated with a unique operator $V \in \mathcal{B}(\mathcal{H}^+, \mathcal{H}^-)$ defined by $v(\Psi,\Phi) = (\Psi, V\Phi)$. Let us now transfer the previous results for form-valued functions to the operator-valued functions $t \in I \mapsto V_t \in \mathcal{B}(\mathcal{H}^+, \mathcal{H}^-)$.
 For convenience, we denote by $\|\cdot\|_{\pm,\mp}$ the operator norm in $\mathcal{B}(\mathcal{H}^\pm, \mathcal{H}^\mp)$.

	\begin{Proposition} \label{prop:op_norm_differentiability}
		Let $\mathcal{H}^+ \subset \mathcal{H} \subset \mathcal{H}^-$ be a scale of Hilbert spaces and~$I\subset\mathbb{R}$ compact. For~$t \in I$, let $v_t: \mathcal{H}^+ \times \mathcal{H}^+ \to \mathbb{C}$ be a family of Hermitian, bounded, sesquilinear forms and~$V(t): \mathcal{H}^+ \to \mathcal{H}^-$ be defined as the unique operators such that $(\Psi, V(t)\Phi) = v_t(\Psi, \Phi)$ for all $\Psi,\Phi\in \mathcal{H}^+$.
		Then:
		\begin{enumerate}[label=\textit{(\roman*)},nosep,leftmargin=*,ref=\ref{prop:op_norm_differentiability}\textit{(\roman*)}]
			\item \label{prop:form-norm}
			For every $t \in I$, the~operator norm of $V(t) \in \mathcal{B}(\mathcal{H}^+,\mathcal{H}^-)$ satisfies:
 \begin{equation*}
				\|V(t)\|_{+,-} = \sup_{\Psi, \Phi \in \mathcal{H}^+ \setminus \{0\}} \frac{|v_t(\Psi, \Phi)|}{\|\Psi\|_+ \|\Phi\|_+}.
			\end{equation*}
		\end{enumerate}
		If, in~addition, the~map $t \mapsto v_t(\Psi, \Phi)$ is continuously differentiable for every $\Psi, \Phi \in \mathcal{H}^+$, then:
		\begin{enumerate}[resume,label=\textit{(\roman*)},nosep,leftmargin=*,ref=\ref{prop:op_norm_differentiability}\textit{(\roman*);}]
			\item \label{prop:form-norm-continuity}
			The map $t \mapsto V(t)$ is continuous with respect to the $\|\cdot\|_{+,-}$ norm, and the family $\{V(t)\}_{t \in I}$ is uniformly bounded;
			% That is, $ \sup_{t \in I} \|V(t)\|_{+,-} < M$.
			\item \label{prop:form-differentiability}
			The derivative $\frac{\d}{\d t} V(t)$ exists in the $\|\cdot\|_{+,-}$-norm sense, and it is uniformly bounded, that is there is a constant $K$ such that for every $t \in I$,
			\begin{equation*}
				\left\|\frac{\d}{\d t} V(t)\right\|_{+,-} \leq K.
			\end{equation*}
		\end{enumerate}
	\end{Proposition}
	\begin{proof}
		By the definition of the operator norm, one has that:
		\begin{equation*}
			\|V(t)\|_{+,-} = \sup_{\Phi \in \mathcal{H}^+ \setminus \{0\}} \frac{\|V(t) \Phi\|_{-}}{\|\Phi\|_{+}}
			= \sup_{\substack{\Phi \in \mathcal{H}^+ \setminus \{0\} \\ \Psi \in \mathcal{H}^- \setminus \{0\}}} \frac{|\langle \Psi, V(t) \Phi \rangle_-|}{\|\Psi\|_-\|\Phi\|_{+}},
		\end{equation*}
		where we used the equality $\|\xi\|_- = \sup_{\Psi \neq 0} \frac{|\langle \Psi, \xi \rangle_-|}{\|\Psi\|_-}$, which holds for any $\xi \in \mathcal{H}^-$.

		Using the isomorphism $J: \mathcal{H}^- \to \mathcal{H}^+$ and the pairing $(\cdot, \cdot)$ associated with the scale of Hilbert spaces $\mathcal{H}^+ \subset \mathcal{H} \subset \mathcal{H}^-$ (cf.\ Subsection~\ref{sec:hilbert_scales}), the~previous equation can be written as:
		\begin{equation*}
			\|V(t)\|_{+,-} = \sup_{\substack{\Phi \in \mathcal{H}^+ \setminus \{0\} \\ \Psi \in \mathcal{H}^- \setminus \{0\}}} \frac{|(J\Psi, V(t)\Phi)|}{\|\Psi\|_-\|\Phi\|_+}
			= \sup_{\Psi, \Phi \in \mathcal{H}^+ \setminus \{0\}} \frac{|(\Psi, V(t)\Phi)|}{\|\Psi\|_+ \|\Phi\|_+},
		\end{equation*}
		where we used the fact that $J$ is an isometric bijection from $\mathcal{H}^-$ to $\mathcal{H}^+$.
		By the definition of $V(t)$, we therefore have:
		\begin{equation*}
			\|V(t)\|_{+,-} = \sup_{\Psi, \Phi \in \mathcal{H}^+ \setminus \{0\}} \frac{|v_t(\Psi, \Phi)|}{\|\Psi\|_+ \|\Phi\|_+},
		\end{equation*}
		which proves \emph{(i)}. Property \emph{(ii)} follows immediately from Proposition \ref{prop:form_equicont_equidiff} and \emph{(i)}.

		To prove \emph{(iii)}, note that, by~Lemma \ref{lemma:locally_unif_bound}, $\dot{v}_t(\Psi, \Phi) \coloneqq \frac{\d}{\d t} v_t(\Psi, \Phi)$ defines a bounded sesquilinear form in $\mathcal{H}^+$.
		Let $\dot{V}(t)$ be the unique operator in $\mathcal{B}(\mathcal{H}^+, \mathcal{H}^-)$ such that $(\Psi, \dot{V}(t)\Phi) = \dot{v}_t(\Psi,\Phi)$ for every $\Psi,\Phi \in \mathcal{H}^+$.
		We claim that $\dot{V}(t) = \frac{\d}{\d t} V(t)$ in the norm sense of $\mathcal{B}(\mathcal{H}^+,\mathcal{H}^-)$, that is,
		\begin{equation*}
			\lim_{t \to t_0} \left\| \frac{V(t) - V(t_0)}{t - t_0} - \dot{V}(t_0) \right\|_{+,-} = 0.
		\end{equation*}
		By \emph{(i)}, this is equivalent to:
		\begin{equation*}
			\lim_{t \to t_0} \sup_{\substack{\Psi, \Phi \in \mathcal{H}^+ \\ \|\Psi\|_+ = 1 = \|\Phi\|_+}}
			\left|\frac{v_t(\Psi, \Phi) - v_{t_0}(\Psi, \Phi)}{t - t_0} - \dot{v}_{t_0}(\Psi, \Phi)\right| = 0,
		\end{equation*}
		which holds by Proposition \ref{prop:equidifferentiability_forms}, as~does the uniform bound.
	\end{proof}

	\section{Existence of~Dynamics} \label{sec:existence_dynamics}
	\subsection{Quantum Dynamics and Sesquilinear~Forms} \label{subsec:quantum_dynamics}
 We begin this section with some~definitions.

	\begin{Definition}\label{def:strong_Schrodinger}
		Let $I\subset \mathbb{R}$ be a compact interval, and let $H(t)$, $t\in I$, be a time-dependent Hamiltonian.
		The {(strong) Schrödinger equation} is the linear evolution equation:
		\begin{equation*}
			\frac{\d}{\d t}\Psi(t) = -iH(t)\Psi(t).
		\end{equation*}
	\end{Definition}
	Despite this equation being a linear ordinary differential equation, the~existence and uniqueness of solutions is not guaranteed in the most general case of a time-dependent Hamiltonian with a time-varying domain. Notice that it is implicitly imposing to the solutions that $\Psi(t)\in\dom H(t)$ for all values of $t$. This condition makes it hard to construct explicit sequences that approximate~them.

	As stated in the Introduction, we consider the case of a time-dependent Hamiltonian with a constant form domain.
	\begin{Definition} \label{def:hamiltonian_const_form}
		Let $I \subset \mathbb{R}$ be a compact interval and $\mathcal{H}^+$ a dense subspace of $\mathcal{H}$, and~let $H(t)$, ${t \in I}$, be a family of self-adjoint operators on $\mathcal{H}$ such that, for~any $t\in I$, the~operator $H(t)$ is densely defined on $\dom H(t)$.
		We say that $H(t)$, ${t \in I}$, is a {time-dependent Hamiltonian with constant form domain} $\mathcal{H}^+$ if:
		\begin{enumerate}[label=\textit{(\roman*)},nosep,leftmargin=*]
			\item There is $m >0$ such that $\langle \Phi, H(t)\Phi \rangle \geq -m \|\Phi\|^2$, for~every $\Phi\in\dom H(t)$ and $t\in I$;
			\item For any $t \in I$, the~domain of the Hermitian sesquilinear form $h_t$ associated with $H(t)$ (cf.\ Theorem~\ref{thm:repKato}) is $\mathcal{H}^+$.
		\end{enumerate}
	\end{Definition}
	For the case in which the closed sesquilinear forms associated with $H(t)$ via Theorem~\ref{thm:repKato} have a constant domain, one can define a weak version of the Schrödinger equation.
	\begin{Definition}\label{def:weak_Schrodinger}
		Let $I\subset \mathbb{R}$ be a compact interval, and let $H(t)$, $t\in I$, be a time-dependent Hamiltonian with constant form domain $\mathcal{H}^+$.
		For $t\in I$, let $h_t$ be the sesquilinear form uniquely associated with the self-adjoint operator $H(t)$.
		The {weak Schrödinger equation} is the linear evolution equation:
		\begin{equation*}
			\frac{\d}{\ dt} \langle \Phi, \Psi(t)\rangle = -ih_t(\Phi,\Psi(t)),
 \quad \text{for every } \Phi\in \mathcal{H}^+.
		\end{equation*}
	\end{Definition}
	The strong Schrödinger equation is the usual Schrödinger equation. We use the adjective \emph{strong} when we want to emphasise the difference with the weak Schrödinger equation.
	It is immediate to prove that solutions of the strong Schrödinger equation are also solutions of the weak Schrödinger equation when the latter is~defined.

	The solutions of the initial-value problems for the weak and strong Schrödinger equations are given in terms of what are known as unitary propagators.
	\begin{Definition}\label{def:unitarypropagator}
		A {unitary propagator} is a two-parameter family of unitary operators $U(t,s)$, $s,t\in \mathbb{R}$, satisfying:
		\begin{enumerate}[label=\textit{(\roman*)},nosep,leftmargin=*]
			\item $U(t,s)U(s,r)= U( t , r)$;
			\item $U(t,t)= \mathbb{I}$;
			\item $U(t,s)$ is jointly strongly continuous in $t$ and $s$.
		\end{enumerate}
	\end{Definition}
	\begin{Definition}
  We say that the strong (weak) Schrödinger equation is {solvable} if there exists a unitary propagator $U(t,s)$, $t,s\in\mathbb{R}$, such that, for~any $\Psi \in \dom H(s)$ ($\Psi\in \mathcal{H}^+$), the~equation with initial value $\Psi$ at time $s\in\mathbb{R}$ admits the curve $\Psi(t) := U(t,s)\Psi$ as its unique solution.
	\end{Definition}

Following the discussion in Section~\ref{sec:hilbert_scales}, given a time-dependent Hamiltonian ${H(t)}$, ${t \in I}$, with~constant form domain $\mathcal{H}^+$, for~each fixed time $t$, we can associate $H(t)$ with a scale of Hilbert spaces:
	\begin{equation*}
	(\mathcal{H}^+, \langle \cdot, \cdot \rangle_{+,t}) \subset
	(\mathcal{H}, \langle \cdot, \cdot \rangle) \subset
	(\mathcal{H}^-_t, \langle \cdot, \cdot \rangle_{-,t}),
\end{equation*}
where $\langle \Psi, \Phi \rangle_{\pm,t} := \langle (H(t) + m + 1)^{\pm\sfrac{1}{2}}\Psi, (H(t) + m + 1)^{\pm\sfrac{1}{2}} \Phi \rangle$, and $\mathcal{H}^-_t$ denotes the closure of $\mathcal{H}$ with respect to the norm defined by $\|\Phi\|_{-,t}^2 \coloneqq \langle \Phi, \Phi \rangle_{-,t}$.
We denote $\mathcal{H}^+_t = (\mathcal{H}^+, \langle \cdot, \cdot \rangle_{+,t})$, and with~a slight abuse of notation, we use the symbol $\mathcal{H}^-_t$ to represent the Hilbert space $(\mathcal{H}^-_t, \langle \cdot, \cdot \rangle_{-,t})$.
We denote by $(\cdot,\cdot)_t : \mathcal{H}^+_t \times \mathcal{H}^-_t \cup \mathcal{H}^-_t \times \mathcal{H}^+_t \to \mathbb{C}$ the canonical pairings.
The unbounded operators $H(t)$ with a constant form domain can be continuously extended to bounded operators $\tilde{H}(t): \mathcal{H}^+_t \to \mathcal{H}^-_t$.
When they are defined, the~inverse operators, $H(t)^{-1} \in \mathcal{B}(\mathcal{H})$, can be extended to bounded operators from $\mathcal{H}^-_t$ to $\mathcal{H}^+_t$, which coincide with $\tilde{H}(t)^{-1}$.
As already done in Section~\ref{sec:hilbert_scales}, in~order to simplify the notation, we drop the tilde, denoting the extensions and the original operators with the same symbols.
We also denote by $\|\cdot\|_{\pm,\mp,t}$ the norm in $\mathcal{B}(\mathcal{H}^\pm_t, \mathcal{H}^\mp_t)$.

Scales of Hilbert spaces are the key objects needed to prove the main results in both J.~Kisyński's and B.~Simon's seminal works, which we shall discuss in the next~subsections.

	\subsection{B.~Simon's~Approach} \label{subsec:simon}
	Let us now introduce the main ideas used by B.~Simon in~\cite{Simon1971}.
	The following assumptions for proving the existence of dynamics were made.
	\begin{Assumption} \label{assump:simon}
		Let $H_0$ be a positive operator on a Hilbert space, and let $\mathcal{H}^+ \subset \mathcal{H} \subset \mathcal{H}^-$ be its associated scale of Hilbert spaces.
		For $0 \leq t \leq T$, let $H(t)$ be a family of bounded operators from $\mathcal{H}^+$ to $\mathcal{H}^-$ such that $(\Psi,H(t)\Phi) = (H(t)\Psi,\Phi)$ for all $\Psi,\Phi\in\mathcal{H}^+$.
		We assume that there is $C>0$ independent of $t$ such~that:
		\begin{enumerate}[label=\textit{(S\arabic*)},nosep,leftmargin=*]
			\item \label{assump:simon_normequiv}
			$C^{-1} (H_0 + 1) \leq H(t) \leq C(H_0 + 1)$;
			\item \label{assump:simon_regularity}
			$B(t) = \frac{\d}{\d t} H(t)^{-1}$ exists in the sense of the norm of $\mathcal{H}$ and~satisfying:
			\begin{equation*}
				\|H(t)^{\sfrac{1}{2}} B(t) H(t)^{\sfrac{1}{2}}\| \leq C.
			\end{equation*}
		\end{enumerate}
	\end{Assumption}

	The approach by B.~Simon defines the Hamiltonian of the system through a family of operators from $\mathcal{H}^+$ to $\mathcal{H}^-$.
	This simplifies some usual problems for the existence of quantum dynamics.
 On the one hand, a~self-adjoint extension of $H(t)$, as~an unbounded operator on $\mathcal{H}$, is fixed (see \cite{Simon1971}, Lemma II.6).
 On the other hand, it simplifies the problem raised by the time dependence of $\dom H(t)$.
 Having a common domain as operators from $\mathcal{H}^+$ to $\mathcal{H}^-$ sidesteps some of the technical difficulties that appear in this~case. 

	Under these hypothesis, B.~Simon proved the existence of dynamics in \cite{Simon1971}, Appendix~II.7.
	\begin{Theorem}[\cite{Simon1971}, Thm. II.27] \label{thm:simon}
		Let $H(t)$ satisfy Assumption \ref{assump:simon}.
		Then, for~any $\Phi_0 \in \mathcal{H}^+$, there is a unique function $\Phi(t) \in \mathcal{H}^+$ such~that:
		\begin{enumerate}[label=\textit{(\roman*)},nosep,leftmargin=*]
			\item $\Phi(t)$ is continuous in the $\mathcal{H}^+$ weak topology, i.e.,\ for all $\Psi \in \mathcal{H}^-$, $t \mapsto \langle \Psi, \Phi(t) \rangle$ is continuous;
			\item For any $\Psi \in \mathcal{H}^+$, $\Phi(t)$ solves the weak Schrödinger equation (cf.\ Definition~\ref{def:weak_Schrodinger}) with initial condition $\Phi_0$:
			\begin{equation*}
				\frac{\d}{\d t}\langle \Psi, \Phi(t) \rangle = -i ( \Psi, H(t)\Phi(t) ), \qquad
				\Phi(0) = \Phi_0.
			\end{equation*}
		\end{enumerate}
		Moreover:
		\begin{enumerate}[resume,label=\textit{(\roman*)},nosep,leftmargin=*]
			\item $\displaystyle \lim_{t \to t_0} \left\| \frac{\Phi(t) - \Phi(t_0)}{t - t_0} + i H(t_0)\Phi(t_0) \right\|_- = 0$;
			\item $\|\Phi(t)\| = \|\Phi_0\|$;
			\item $\Phi(t)$ is $\|\cdot\|$-continuous.
		\end{enumerate}
		Thus, the~map $U(t,s): \Phi(s)\in\mathcal{H}^+ \mapsto \Phi(t)\in\mathcal{H}^+$ is unitary on $\mathcal{H}$ and can be extended to a unitary propagator.
	\end{Theorem}

	The proof of B.~Simon is based on K. Yosida's idea (\cite{Yosida1965}, pp. 425--429), whose original theorem only applies to the case in which $\dom H(t)$ does not depend on $t$.
	The proof is based on the introduction of an approximating family of Hamiltonians $H_n(t) = H(t) (1 + n^{-1}H(t))^{-1}$, which can be seen as a bounded operator in $\mathcal{B}(\mathcal{H}^-)$;
	the existence of unitary propagators $U_n(t,s)$ for $H_n(t)$ is provided by a Dyson expansion.
	Using these propagators, for~any $\Phi_0 \in \mathcal{H}^+$, approximated solutions $\Phi_n(t) = U_n(t,s)\Phi_0$ for the dynamical equation for $H(t)$ are constructed, and~it is shown that Assumption \ref{assump:simon} ensures that $\{\Phi_n(t)\}_{n \in \mathbb{N}} \subset \mathcal{H}^+$ has a convergent subsequence, whose limit satisfies the properties of the~theorem.

	\subsection{J.~Kisyński's~Approach} \label{subsec:kisynski}
	In~\cite{Kisynski1964}, J.~Kisyński studied the existence of solutions for the equation:
\begin{equation} \label{eq:kisynski_problem}
		\frac{\d}{\d t} \Phi(t) = A(t) \Phi(t), \quad \Phi(0) = \Phi_0,
	\end{equation}
	on a Banach space, with~$A(t)$ being an unbounded linear operator defined on a dense domain $\dom A(t)$, which can depend on time.
	He first proved the existence of solutions for the case in which $\dom A(t)$ is constant by using Yosida's approximation; then, he provided sufficient conditions for the existence of solutions in the more general case in which $\dom A(t)$ does depend on $t$. These general results are finally applied to the case in which $A(t) = -i \Lambda(t)$, with~$\Lambda(t)$ being a positive, self-adjoint operator on a Hilbert space, such that $\dom \Lambda(t)^{\sfrac{1}{2}}$ is independent of $t$ (\cite{Kisynski1964}, Secs. 7 and 8).

	In order to prove the existence of a unitary propagator for this case, J.~Kisyński assumed the following.
	\begin{Assumption}[\cite{Kisynski1964}, Hyp. 7.1] \label{assump:kisynski}
		Let $\mathcal{H}$ be a Hilbert space with inner product $\langle \cdot, \cdot \rangle$, and~let $\mathcal{H}^+ \subset \mathcal{H}$ be a dense~subspace:
		\begin{enumerate}[label=\textit{(K\arabic*)},nosep,leftmargin=*]
			\item \label{assump:kisynski_normequiv}
			For every $t \in [0,T]$, $\langle \cdot, \cdot \rangle_{+,t}$ is an inner product on $\mathcal{H}^+$ endowing it with the structure of a Hilbert space $\mathcal{H}^+_t$, which is continuously contained on $\mathcal{H}$;
			\item \label{assump:kisynski_regularity}
			For every $t \in [0,T]$ and every $\Psi,\Phi \in \mathcal{H}^+$, the~function $t \mapsto \langle \Psi, \Phi \rangle_{+,t}$ is in $C^n(0,T)$ for $n \geq 1$.
		\end{enumerate}
	\end{Assumption}

	Under these hypotheses, for~each $t \in [0,T]$, J.~Kisyński constructed the scale of Hilbert spaces $\mathcal{H}^+_t \subset \mathcal{H} \subset \mathcal{H}^-_t$ associated with $\langle \Psi, \Phi \rangle_{+,t}$ (see \cite{Kisynski1964}, Lemmas 7.2--7.6).
	Then, the~following operators are defined (cf.\ \cite{Kisynski1964}, Lemmas 7.7--7.10).
	\begin{Proposition} \label{prop:kisynski_operators}
		Under Assumption \ref{assump:kisynski}, for~every $t \in [0,T]$, the~following statements~hold:
		\begin{enumerate}[label=\textit{(\roman*)},nosep,leftmargin=*]
			\item Define the domain:
			\begin{equation*}
				\mathcal{D}(\Lambda_0) = \left\{\Phi \in \mathcal{H}^+ \mid \sup\{|\langle \Psi, \Phi \rangle_{+,t}| \mid \Psi \in \mathcal{H}^+, \|\Psi\|=1\} <\infty \right\}.
			\end{equation*}
			Then, the equality:
			\begin{equation*}
				\langle \Psi, \Lambda_0(t)\Phi \rangle = \langle \Psi, \Phi \rangle_{+,t}, \qquad
				\Phi \in \mathcal{D}(\Lambda_0), \quad \Psi \in \mathcal{H}^+,
			\end{equation*}
			defines a self-adjoint, positive unbounded operator on $\mathcal{H}$;
			\item Let $\Lambda(t)$ be the closure of $\Lambda_0(t)$ on $\mathcal{H}^-_t$; then, $\Lambda(t)$ is a self-adjoint, positive operator on $\mathcal{H}^-_t$ with $\dom \Lambda(t) = \mathcal{H}^+$.\
			Moreover, $\Lambda(t)^{-1}$ is the canonical isomorphism $J: \mathcal{H}^-_t \to \mathcal{H}^+_t$ (cf.\ Section~\ref{sec:hilbert_scales}) for the scale of Hilbert spaces $\mathcal{H}^+_t \subset \mathcal{H} \subset \mathcal{H}^-_t$;
			\item Let $\mathcal{D}(\Lambda_1) = \{\Phi \in \mathcal{D}(\Lambda_0) \mid \Lambda_0\Phi \in \mathcal{H}^+\}$ and $\Lambda_1\Phi = \Lambda_0\Phi$ for $\Phi \in \mathcal{D}(\Lambda_1)$.
			Then, $\Lambda_1$ is a self-adjoint, positive unbounded operator on $\mathcal{H}^+_t$.
		\end{enumerate}
	\end{Proposition}

	Finally, in~\cite{Kisynski1964}, Sec. 8, J.~Kisyński showed the following theorem on the existence of unitary propagators.
	Starting his construction from a family of inner products (which are sesquilinear forms), he constructed the associated scales of Hilbert spaces and~then focused on the operators representing these inner products.
	This approach has the same advantages as B.~Simon's: the self-adjointness of the operators is guaranteed by the hermiticity of the inner products and the representation theorem (see Theorem~\ref{thm:repKato}).
	Additionally, it sets the problem on a constant (form) domain $\mathcal{H}^+$.

	\begin{Theorem}[\cite{Kisynski1964}, Thm. 8.1] \label{thm:kisynski_original}
		Under Assumption \ref{assump:kisynski} with $n \geq 1$, there is a unique unitary propagator $U(t,s)$ for the problem \eqref{eq:kisynski_problem} such~that:
		\begin{enumerate}[label=\textit{(\roman*)},nosep,leftmargin=*]
			\item The continuous extension of $U(t,s)$ to $\mathcal{B}(\mathcal{H}^-)$ is strongly continuous for $s,t \in [0,T]$;
			\item $U(t,s) \mathcal{H}^+ = \mathcal{H}^+$ and the restriction $U(t,s) \in \mathcal{B}(\mathcal{H}^+)$ is strongly continuous for $s,t \in [0,T]$;
			\item $t,s \mapsto U(t,s)$ is a strongly continuously differentiable function with values on $\mathcal{B}(\mathcal{H}^+, \mathcal{H}^-)$ for $t,s \in [0,T]$, and~in this sense:
			\begin{equation*}
				\frac{\d}{\d t} U(t,s) = -i\Lambda(t)U(t,s) \quad\text{and}\quad
				\frac{\d}{\d s} U(t,s) = -iU(t,s)\Lambda(s).
			\end{equation*}
		\end{enumerate}
		Moreover, if~Assumption \ref{assump:kisynski} is satisfied with $n \geq 2$, it also~holds that:
		\begin{enumerate}[resume,label=\textit{(\roman*)},nosep,leftmargin=*]
			\item $U(t,s)\mathcal{D}(\Lambda_1(s)) = \mathcal{D}(\Lambda_1(t))$ for $t,s \in [0,T]$ and for $s \in [0,T]$ and $\Phi \in \mathcal{D}(\Lambda_1(s))$ fixed, the~function $t \mapsto U(t,s)\Phi$ is continuously differentiable in $[0,T]$ in the sense of $\mathcal{H}^+$, and there, it satisfies:
			\begin{equation*}
				\frac{\d}{\d t} U(t,s)\Phi = -i \Lambda_1(t)U(t,s)\Phi;
			\end{equation*}
			\item $U(t,s)\mathcal{D}(\Lambda_0(s)) = \mathcal{D}(\Lambda_0(t))$ for $t,s \in [0,T]$ and for $s \in [0,T]$ and $\Phi \in \mathcal{D}(\Lambda_0(s))$ fixed, the~function $t \mapsto U(t,s)\Phi$ is continuously differentiable in $[0,T]$ in the sense of $\mathcal{H}$, and there, it satisfies:
			\begin{equation*}
				\frac{\d}{\d t} U(t,s)\Phi = -i \Lambda_0(t)U(t,s)\Phi.
			\end{equation*}
		\end{enumerate}
	\end{Theorem}

	\subsection{Relations between the Approaches of B.~Simon and J.~Kisyński} \label{subsec:simon_vs_kisynski}

 We start by examining the similarities between the two approaches.
 Even though the starting point for each of them is apparently different, the~tools they use are analogous.
 The link allowing us to relate both approaches are the scales of Hilbert spaces.
 B.~Simon started its construction with the scale of Hilbert spaces associated with an operator $H_0$, $\mathcal{H}^+ \subset \mathcal{H} \subset \mathcal{H}^-$ and~then considered the family of operators $H(t): \mathcal{H}^+ \to \mathcal{H}^-$.
 On the other hand, J.~Kisyński started from a family of inner products $\langle \cdot, \cdot \rangle_{+,t}$ and built the scales of Hilbert spaces $\mathcal{H}^+_t \subset \mathcal{H} \subset \mathcal{H}^-_t$ associated with the sesquilinear forms $\langle \cdot, \cdot \rangle_{+,t}$ and the family of operators $\Lambda(t): \mathcal{H}^+ \to \mathcal{H}^-$ representing them.
 Comparing Theorem~\ref{thm:simon} and Theorem~\ref{thm:kisynski_original}, it is clear that the role the family of operators $H(t)$ plays in the approach by B.~Simon is the same that the family $\Lambda(t)$ plays in J.~Kisyński's.
 Therefore, the~problem addressed by them is the same if $h_t(\cdot,\cdot) \coloneqq \langle \cdot, \cdot \rangle_{+,t}$ is the sesquilinear form associated with $H(t)$.
 For the rest of this section, we assume that this equality~holds.

	Before further studying the relation between both approaches, let us briefly review the assumptions B.~Simon and J.~Kisyński made to prove the existence of a unitary propagator.
	There are two main ingredients playing a central role in their proofs: the regularity of the generators $H(t)$ and the uniform equivalence of the norms $\|\cdot\|_{+,t} \coloneqq \sqrt{\langle \cdot, \cdot \rangle_{+,t}}$, i.e.,~\hbox{$c^{-1}\norm{\cdot}_{+,t'}\leq \norm{\cdot}_{+,t}\leq c \norm{\cdot}_{+,t'}$}.

	Although both of them use repeatedly the equivalence of the norms, neither B.~Simon nor J.~Kisyński imposed it as an explicit assumption; still, the~equivalence \textit{does} follow from their assumptions.
	Let us first examine how this uniform equivalence of the norms appears in each of the~approaches.

	\begin{Proposition} \label{prop:norm_equiv_simon}
		Let $H_0, H(t)$ be as in Assumption \ref{assump:simon}; for $\Phi \in \mathcal{H}^+$, define the norms:
		\begin{equation*}
			\|\Phi\|_{+,t} \coloneqq \sqrt{\left( \Phi, H(t)\Phi \right)}, \qquad
			\|\Phi\|_0 \coloneqq \sqrt{\left( \Phi, (H_0 + 1)\Phi \right)}.
		\end{equation*}
		Then, \ref{assump:simon_normequiv} holds if and only if the norms $\|\cdot\|_{+,t}$ are equivalent to $\|\cdot\|_0$ uniformly on $t \in [0,T]$.
	\end{Proposition}
	\begin{proof}
		It is enough to note that $\|\Phi\|_0^2 = (\Phi, (H_0 + 1)\Phi)$ and $\|\Phi\|_{+,t}^2 = (\Phi, H(t)\Phi)$.
		Therefore, taking square roots in \ref{assump:simon_normequiv}, the~uniform equivalence of the norms follows.
		Conversely, taking squares on the uniform equivalence of the norms yields \ref{assump:simon_normequiv}.
	\end{proof}
	\begin{Remark}
		Note that, since we have $\|\cdot\|_0 \sim \|\cdot\|_{+,t}$ with constant uniform on $t$, for~any $t_0 \in [0,T]$, there is $K$ such that:
		\begin{equation*}
			K^{-1} \|\cdot\|_{+,t_0} \leq \|\cdot\|_{+,t} \leq K \|\cdot\|_{+,t_0},
		\end{equation*}
		for every $t \in [0, T]$.
	\end{Remark}

	\begin{Proposition} \label{prop:norm_equiv_kisynski}
		Let $\langle \cdot, \cdot \rangle_{+,t}$ be inner products as in Assumption \ref{assump:kisynski}.
		Then, \ref{assump:kisynski_normequiv} implies that, for~every $t \in [0,T]$, the~norms $\|\cdot\|_{+,t}$ are equivalent.
		If, moreover, $t \mapsto \langle \cdot, \cdot \rangle_{+,t}$ is in $C^1([0,T])$, then the equivalence of the norms is uniform on $t$.
	\end{Proposition}
	\begin{proof}
		By \ref{assump:kisynski_normequiv}, $\mathcal{H}^+_t$ is continuously embedded in $\mathcal{H}$; that is, there is $K > 0$ such that:
		\begin{equation*}
			\|\Phi\| \leq K \|\Phi\|_{+,t}, \qquad \forall \Phi \in \mathcal{H}^+_t.
		\end{equation*}
		Therefore, $\langle \cdot, \cdot \rangle_{+,t}$, as~a Hermitian sesquilinear form densely defined on $\mathcal{H}$, is strictly positive, and there exists a self-adjoint, strictly positive operator $A(t): \mathcal{H}^+ \to \mathcal{H}$ such that:
		\begin{equation*}
			\langle \Psi, \Phi \rangle_{+,t} = \langle A(t)\Psi, A(t)\Phi \rangle, \qquad
			(\Psi,\Phi \in \mathcal{H}^+).
		\end{equation*}
		Since it is positive and self-adjoint, $A(t)^{-1}: \mathcal{H} \to \mathcal{H}^+$ is a bounded self-adjoint operator.
		By the closed graph theorem, $A(t')A(t)^{-1}: \mathcal{H} \to \mathcal{H}$ is a bounded operator, and for $\Phi \in \mathcal{H}^+$:
		\begin{equation*}
			\|\Phi\|_{+,t'} = \|A(t')A(t)^{-1}A(t)\Phi\|
			\leq C_{t,t'} \|A(t)\Phi\| = C_{t,t'}\|\Phi\|_{+,t},
		\end{equation*}
		where $C_{t,t'} \coloneqq \|A(t')A(t)^{-1}\|$.

		The rest of the proof follows \cite{Kisynski1964}, Lemma 7.3, where it is proven using the Uniform Boundedness Principle that the constant can be chosen independently of $t$ and $t'$.
	\end{proof}

	\begin{Corollary} \label{corol:kisynski_implies_S1}
		Let $\langle \cdot, \cdot \rangle_{+,t}$, $t \in [0, T]$, be inner products satisfying Assumption \ref{assump:kisynski}, and~let $H(t)$ be the operator in $\mathcal{B}(\mathcal{H}^+, \mathcal{H}^-)$ such that $\langle \Psi, \Phi \rangle_{+,t} = (\Psi, H(t)\Phi)$ for every $\Psi, \Phi \in \mathcal{H}^+$.
		Then, $H(t)$ satisfies \ref{assump:simon_normequiv} of Assumption \ref{assump:simon} with $H_0 = H(t_0)$ for some $t_0$ fixed.
	\end{Corollary}
	\begin{proof}
		Let $\|\cdot\|_{+,t} = \sqrt{\langle \cdot, H(t) \cdot \rangle}$. By~Proposition \ref{prop:norm_equiv_kisynski}, for~any $t,t_0$, we have:
		\begin{equation*}
			K^{-1} \|\Phi\|_{+,t_0} \leq \|\Phi\|_{+,t} \leq K \|\Phi\|_{+,t_0}.
		\end{equation*}
		For a fixed $t_0$, define the norm $\|\Phi\|_0 \coloneqq \sqrt{\|\Phi\|_{+,t_0}^2 + \|\Phi\|^2}$. Let us show that $\|\cdot\|_0$ is equivalent to $\|\cdot\|_{+,t}$.
		Since $\|\Phi\|_{+,t_0} \leq \|\Phi\|_0$, it follows that $\|\Phi\|_{+,t} \leq K \|\Phi\|_0$.
		Furthermore, since $\|\Phi\|_{+,t} \geq \|\Phi\|$,
		\begin{equation*}
			K^{-2} \|\Phi\|_0^2 \leq K^{-2}\|\Phi\|_{+,t_0}^2 + \|\Phi\|^2 \leq 2 \|\Phi\|_{+,t}^2.
		\end{equation*}
		Hence, there is $\tilde{K} > 1$ such that:
		\begin{equation*}
			\tilde{K}^{-1} \|\Phi\|_{+,t} \leq \|\Phi\|_0 \leq \tilde{K} \|\Phi\|_{+,t}.
		\end{equation*}
		By Proposition \ref{prop:norm_equiv_simon}, the~preceding inequalities imply \ref{assump:simon_normequiv}.
	\end{proof}

	Using the structure of the scales of Hilbert spaces, from~the equivalence of the norms $\|\cdot\|_{+,t}$, the~equivalence of the norms $\|\cdot\|_{-,t}$ follows as well (see Theorem~\ref{thm:eqiv+_equiv-}).
	This motivates us to work with a reference norm $\|\cdot\|_\pm$ and use the uniform equivalence when a particular $\|\cdot\|_{\pm,t}$ is convenient.
	The most convenient choice is to take $\|\cdot\|_\pm = \|\cdot\|_{\pm,t_0}$ for some reference $t_0$.
	We denote by $\langle \cdot, \cdot \rangle_+$ the inner product such that $\|\cdot\|_+^2 = \langle \cdot, \cdot \rangle_+$ and by $\mathcal{H}^\pm$ the Hilbert space $(\mathcal{H}^\pm, \langle \cdot, \cdot \rangle_\pm)$.

	Let us study now the regularity of the generators $H(t)$.

	\begin{Proposition} \label{prop:regularity_simon}
		Let $H(t)$ be as in Assumption \ref{assump:simon}.
		Then, \ref{assump:simon_regularity} holds if and only if $t \mapsto H(t)$ is differentiable in the sense of $\mathcal{B}(\mathcal{H}^+, \mathcal{H}^-)$ and, for~every $t \in [0, T]$,
		\begin{equation*}
			\left\| \frac{\d}{\d t} H(t) \right\|_{+,-} \leq C.
		\end{equation*}
	\end{Proposition}
	\begin{proof}
		Note that, if~\ref{assump:simon_regularity} holds,
		\begin{equation*}
			\|B(t)\|_{-,+} = \sup_{\Phi \in \mathcal{H}^-} \frac{\|B(t) \Phi\|_+}{\|\Phi\|_-}
			\leq K \sup_{\Psi \in \mathcal{H}} \frac{\|H(t)^{\sfrac{1}{2}} B(t) H(t)^{\sfrac{1}{2}} \Psi\|}{\|\Psi\|}
			\leq KC
		\end{equation*}
		where we used the equivalence of the norms $\|\cdot\|_{\pm,t} \sim \|\cdot\|_\pm$ (cf.\ Proposition~\ref{prop:norm_equiv_simon}) and the assumption $\|H(t)^{\sfrac{1}{2}} B(t) H(t)^{\sfrac{1}{2}}\| < C$.
		Therefore, $\|B(t)\|_{-,+} < KC$, so that $B(t)$ can be continuously extended to an operator in $\mathcal{B}(\mathcal{H}^-, \mathcal{H}^+)$.

		Define the operator $T_\delta(t) = \delta^{-1}[H(t+\delta)^{-1} - H(t)^{-1}] - B(t)$ in $\mathcal{B}(\mathcal{H}^-, \mathcal{H}^+)$. It follows that:
		\begin{equation*}
			\|T_\delta(t)\|_{-,+} = \sup_{\Phi \in \mathcal{H}^-} \frac{\|T_\delta(t)\Phi\|_+}{\|\Phi\|_-}
			= \|A_0^{\sfrac{1}{2}} T_\delta(t) A_0^{\sfrac{1}{2}}\|,
		\end{equation*}
		where $A_0$ is the strictly positive self-adjoint operator such that $\|\cdot\|_\pm = \|A_0^{\pm\sfrac{1}{2}}\cdot\|$.
		Hence, $H(t)^{-1}$ is differentiable in the sense of $\|\cdot\|_{-,+}$ if $H(t)^{-1}$ is differentiable in the sense of $\|\cdot\|$. By~the product rule, this is equivalent to $H(t)$ being differentiable in the sense of $\|\cdot\|_{+,-}$.

		Conversely, assume $H(t)$ is differentiable in the sense of $\mathcal{B}(\mathcal{H}^+, \mathcal{H}^-)$.
		By the product rule, $H(t)^{-1}$ is differentiable in the sense of $\mathcal{B}(\mathcal{H}^-, \mathcal{H}^+)$.
		Denote by $B(t)$ the derivative of $H(t)^{-1}$ in the sense of $\mathcal{B}(\mathcal{H}^-, \mathcal{H}^+)$, and consider again the operator $T_\delta(t) = \delta^{-1}[H(t+\delta)^{-1} - H(t)^{-1}] - B(t)$.
		For $\Phi \in \mathcal{H}$, we have:
		\begin{equation*}
			\|T_\delta(t)\Phi\| \leq \|T_\delta(t)\Phi\|_+ \leq \|T_\delta(t)\|_{-,+} \|\Phi\|_- \leq \|T_\delta(t)\|_{-,+}\|\Phi\|.
		\end{equation*}
		Therefore, $\|T_\delta(t)|_\mathcal{H}\| \leq \|T_\delta(t)\|_{-,+}$, and the differentiability of $H(t)^{-1}$ in the sense of $\mathcal{B}(\mathcal{H}^-, \mathcal{H}^+)$ implies the differentiability in the sense of $\mathcal{B}(\mathcal{H})$.

		Finally,
		\begin{equation*}
			H(t)^{\sfrac{1}{2}} \frac{\d}{\d t} H(t)^{-1} H(t)^{\sfrac{1}{2}} = H(t)^{-\sfrac{1}{2}} \frac{\d}{\d t} H(t) H(t)^{-\sfrac{1}{2}},
		\end{equation*}
		and it follows that $\|\frac{\d}{\d t} H(t)\|_{+,-,t} = \|H(t)^{\sfrac{1}{2}} B(t) H(t)^{\sfrac{1}{2}}\|$.
		The equivalence of the norms $\|\cdot\|_{\pm,t} \sim \|\cdot\|_{\pm}$ implies that $\|H(t)^{\sfrac{1}{2}} B(t) H(t)^{\sfrac{1}{2}}\|$ is bounded uniformly on $t$ if and only if $\|\frac{\d}{\d t} H(t)\|_{+,-}$ is bounded uniformly on $t$.
	\end{proof}
	\begin{Remark} \label{remark:op_deriv_plusminus}
		Note that, in~proving Proposition \ref{prop:regularity_simon}, we also showed~that:
		\begin{enumerate}[label=\textit{(\roman*)},nosep,leftmargin=*]
			\item $\frac{\d}{\d t} H(t)^{-1}$ exists in the sense of $\mathcal{B}(\mathcal{H})$ if and only if it exists in the sense of $\mathcal{B}(\mathcal{H}^-, \mathcal{H}^+)$, which holds if and only if $\frac{\d}{\d t}H(t)$ exists in the sense of $\mathcal{B(\mathcal{H}^+, \mathcal{H}^-)}$;
			\item $\|H(t)^{\sfrac{1}{2}} \frac{\d}{\d t} H(t)^{-1} H(t)^{\sfrac{1}{2}}\| = \|H(t)^{-\sfrac{1}{2}} \frac{\d}{\d t} H(t) H(t)^{-\sfrac{1}{2}}\|$.
		\end{enumerate}
	\end{Remark}

	\begin{Proposition} \label{prop:regularity_kisynski}
		Let $\langle \cdot, \cdot \rangle_{+,t}$ be inner products as in Assumption \ref{assump:kisynski}, and let $H(t): \mathcal{H}^+ \to \mathcal{H}^-$ be the family of bounded operators defined by $(\Psi, H(t)\Phi) = \langle \Psi, \Phi \rangle_+$.
		Then, \ref{assump:kisynski_regularity} implies that $H(t)$ is differentiable in the sense of $\mathcal{B}(\mathcal{H}^+, \mathcal{H}^-)$ and that there is $C > 0$ such that:
		\begin{equation*}
			\left\| \frac{\d}{\d t} H(t) \right\|_{+,-} \leq C.
		\end{equation*}
	\end{Proposition}
	\begin{proof}
		This is Proposition \ref{prop:form-differentiability} applied to $v_t(\Psi, \Phi) = \langle \Psi, \Phi \rangle_+$, which is continuously differentiable by assumption (see \ref{assump:kisynski_regularity}).
	\end{proof}

	Corollary \ref{corol:kisynski_implies_S1} and Propositions \ref{prop:regularity_simon} and \ref{prop:regularity_kisynski} yield immediately the following results showing the connection between B.~Simon's and J.~Kisyński's approaches. First of all, J.~Kisyński's assumptions imply B.~Simon's assumptions:
	\begin{Theorem} \label{thm:kisynski_implies_simon}
		Let $\langle \cdot, \cdot \rangle_{+,t}$ be a family of inner products satisfying Assumption \ref{assump:kisynski}, and let $H(t): \mathcal{H}^+ \to \mathcal{H}^-$ be the family of bounded operators defined by $(\Psi, H(t)\Phi) = \langle \Psi, \Phi \rangle_{+,t}$.
		Then, the~family $H(t)$ satisfies Assumption \ref{assump:simon} with $H_0 = H(t_0)$ for some $t_0 \in [0,T]$.
	\end{Theorem}
	\begin{proof}
		By \ref{assump:kisynski_normequiv}, $H(t)$ are strictly positive.
		By Corollary \ref{corol:kisynski_implies_S1}, \ref{assump:simon_normequiv} holds.
		Proposition \ref{prop:regularity_kisynski} implies that $H(t)$ is differentiable in the norm sense of $\mathcal{B}(\mathcal{H}^+, \mathcal{H}^-)$, with~$\|\frac{\d}{\d t}H(t)\|_{+,-} <C$, and~by Proposition \ref{prop:regularity_simon}, this implies \ref{assump:simon_regularity}.
	\end{proof}
	\begin{Remark}
		Note that the regularity obtained for $H(t)$ is higher than required by \ref{assump:simon_regularity}: not only is $t \mapsto H(t)$ differentiable, but~it is also continuously differentiable.
		This is implied by Proposition~\ref{prop:op_norm_differentiability} and the fact that $t \mapsto \langle \cdot, \cdot \rangle_{+,t}$ is continuously differentiable.
		However, this point is crucial since, without~the derivative $\frac{\d}{\d t} \langle \cdot, \cdot \rangle_{+,t}$ being continuous, the~uniform bound for $\frac{\d}{\d t}H(t)^{-1}$ would not be recovered.
	\end{Remark}

	The converse implication (i.e.,\ B.~Simon's assumptions imply J.~Kisyński's) also holds under the additional requirement of the continuity of $\frac{\d}{\d t} H(t)^{-1}$.
	\begin{Theorem}
		Let $H_0, H(t)$ be operators as in Assumption \ref{assump:simon}.
		If $H(t)$ satisfies Assumption~\ref{assump:simon} and $t \in [0,T] \mapsto B(t) = \frac{\d}{\d t} H(t)^{-1} \in \mathcal{B}(\mathcal{H})$ is continuous, then $\langle \cdot, \cdot \rangle_{+,t} \coloneqq (\cdot, H(t) \cdot)$ defines a family of inner products on $\mathcal{H}^+$ satisfying Assumption \ref{assump:kisynski} with $n = 1$.
	\end{Theorem}
	\begin{proof}
		From Assumption \ref{assump:simon}, $\mathcal{H}^+$ is a dense subspace of $\mathcal{H}$.
		Moreover, $\langle \cdot, \cdot \rangle_0 \coloneqq (\cdot, (H_0 + 1)\cdot)$ endows $\mathcal{H}^+$ with the structure of a Hilbert space topologically embedded in $\mathcal{H}$.
		By assumption, $\langle \cdot, \cdot \rangle_{+,t} \coloneqq (\cdot, H(t)\cdot)$ defines a family of inner products on $\mathcal{H}^+$, and by \ref{assump:simon_normequiv}, they induce on $\mathcal{H}^+$ topologies that are equivalent to the one induced by $\langle \cdot, \cdot \rangle_0$.
		Therefore, $(\mathcal{H}^+, \langle \cdot, \cdot \rangle_{+,t})$ are Hilbert spaces topologically embedded in $\mathcal{H}$, and~\ref{assump:kisynski_normequiv} holds.

		{By Proposition \ref{prop:regularity_simon}, \ref{assump:simon_regularity} implies that $H(t)$ is differentiable in $\mathcal{B}(\mathcal{H}^+, \mathcal{H}^-)$.
		By \mbox{Proposition \ref{prop:op_norm_differentiability}},} this implies that the limit:
		\begin{equation*}
			\lim_{t \to t_0} \sup_{\substack{\Psi, \Phi \in \mathcal{H}^+ \\ \|\Psi\|_+ = 1 = \|\Phi\|_+}}
			\frac{\langle\Psi, \Phi\rangle_{+,t} - \langle\Psi, \Phi\rangle_{+,t_0}}{t - t_0}
		\end{equation*}
		exists, which clearly implies that for every $\Psi,\Phi \in \mathcal{H}^+$ fixed, $t \mapsto \langle \Psi, \Phi \rangle_{+,t}$ is differentiable.
		Moreover, if~$H(t)$ is continuously differentiable in $\mathcal{B}(\mathcal{H}^+, \mathcal{H}^-)$, so is the map $t \mapsto \langle \Psi, \Phi \rangle_{+,t}$ for every $\Psi,\Phi \in \mathcal{H}^+$.
	\end{proof}

	\subsection{Example: Particle in a Circle with a Point-like~Interaction} \label{example:particle:circle}

{Consider the Hamiltonian of a free particle moving in the circle with a
point-like interaction. We consider the case in which the strength
of the point-like interaction varies with time and use the methods
developed in the previous section to give necessary conditions for the
existence of the solution of the non-autonomous Schrödinger~equation.

As Hilbert space $\mathcal{H}$, we consider the space of the square integrable
function on the interval $[0,2\pi]$. The~Hamiltonian is given by
the Laplace operator on the self-adjoint domain:
$$ \mathcal{D}_{\alpha} = \{\Phi\in \mathcal{H}^2([0,2\pi]) \mid \Phi(0) = \Phi(2\pi);\quad \Phi'(0)-\Phi'(2\pi) = \alpha \Phi(0) \},$$
where $\mathcal{H}^k([0,2\pi])$ stands for the Sobolev space of order $k$.
Each value of $\alpha$ determines a different self-adjoint extension.
We consider the non-autonomous problem in which $\alpha$ is a
function of time. Notice that the time dependence of the operator
does not appear in its functional form, but only at the boundary
condition. The~family of closed sesquilinear forms associated with this
family of self-adjoint operators is:
$$h_{\alpha}(\Phi,\Psi) = \langle \Phi'(x) , \Psi'(x) \rangle + \alpha \overline{\Phi(0)}\Psi(0) $$
with domain
$\operatorname{dom} h_\alpha = \{ \Phi \in \mathcal{H}^1([0,2\pi]) \mid \Phi(0) = \Phi(2\pi)\}$;
we refer to
~\cite{IbortLledoPerezPardo2015, IbortLledoPerezPardo2015a, BalmasedaDiCosmoPerezPardo2019}
for further~details.

To find the family of operators $H(t): \mathcal{H}^+ \to\mathcal{H}^-$, we need the
trace operator. The~trace operator, $\gamma$, is the operator that
maps a function to its boundary values. It is well-defined, continuous
on $\mathcal{H}^1([0,2\pi])$, and surjective; cf.~\cite{AdamsFournier2003}. Now,
notice that the domains of the forms do not depend on the parameter
$\alpha$, and the family of self-adjoint operators defined above has
constant form domain
$\mathcal{H}^+ := \{ \Phi \in \mathcal{H}^1([0,2\pi]) \mid \Phi(0) = \Phi(2\pi)\}$, which
is a closed subspace of $\mathcal{H}^1([0,2\pi])$ because the trace operator is
continuous on $\mathcal{H}^1([0,2\pi])$. The~boundary term of the sesquilinear
form above can be expressed using the trace operator and the paring
$(\cdot, \cdot)$ in the scale of Hilbert spaces as:
$$\alpha \overline{\Phi(0)}\Psi(0) = (\Phi, \alpha \gamma^\dagger\gamma\Psi),$$
where $\gamma^\dagger$ is the adjoint operator with respect to the
pairing. Hence, the family of time-dependent operators
$H(t): \mathcal{H}^+ \to\mathcal{H}^-$ becomes:
$$H(t)= \bar{\Delta} + \alpha(t)\gamma^\dagger\gamma,$$ where
$\bar{\Delta}$ is the continuous extension of the Laplacian,
$\Delta$, to~$\mathcal{H}^+$.

It is well known that $h_a$ is semibounded from below with a bound
that depends continuously on $\alpha$;
cf.~\cite{IbortLledoPerezPardo2015}. It can be proven that
\ref{assump:simon_normequiv} holds as long as
$\sup_{t\in[0,T]} |\alpha(t)|< M$. Proposition
\ref{prop:regularity_simon} establishes that
\ref{assump:simon_regularity} holds if the derivative of the operators
$H(t)$ in the sense of $\mathcal{B}(\mathcal{H}^+,\mathcal{H}^-)$ exists for all
$t\in[0,T]$ and is uniformly bounded, i.e.,~for any
$\alpha:[0,T]\to \mathbb{R}$ continuous with the bounded derivative. Theorem
\ref{thm:simon} shows that for any such $\alpha$, a unique solution of
the weak Schrödinger equation exists. The~same result is obtained under
the Assumptions \ref{assump:kisynski_normequiv} and
\ref{assump:kisynski_regularity}. Condition
\ref{assump:kisynski_normequiv} is satisfied by construction, and
\ref{assump:kisynski_regularity} is satisfied for $n$ if and only if
$\alpha:[0,T]\to\mathbb{R}$ is in $C^n[0,T]$. Notice that for $n>1$ and
since $[0,T]$ is compact, both the function $\alpha$ and its
derivative are bounded in the interval. Theorem
\ref{thm:kisynski_original} also establishes that a unique solution of
the weak Schrödinger equation exists. One has to require, however, the~stronger condition of the continuous differentiability of the function
$\alpha$. If~$\alpha\in C^2[0,T]$, then Theorem
\ref{thm:kisynski_original} establishes that strong solutions of the
Schrödinger equation exist, i.e.,~solutions of the Schrödinger equation
in the usual sense.
}

\section{Conclusions}

 In this article, after~carefully revising the two seminal approaches of J.~Kisyński and B.~Simon to the well-posedness of the Schrödinger problem associated with a time-varying Hamiltonian, we compared the assumptions at the roots of their results by using Hilbert scales as our primary tool.
We studied the relation between one-parameter families of sesquilinear forms and operators representing them. In~particular, we studied the connection between their respective continuity and differentiability properties. Theses results, presented in Section~\ref{sec:preliminaries}, have their own interest and allow comparing the approaches presented in Section~\ref{sec:existence_dynamics}.

Non-continuous Hamiltonians, in~the sense described in the article, cannot lead to physically meaningful examples in a context with time-dependent domains, i.e.,~time-dependent boundary conditions. In~this context, it is important from the physical point of view to have good characterisations of the well posedness of the Schrödinger equation. The~comparison between the approaches manifests that the conditions imposed by J.~Kisyński are easier to check as it is easier to check the continuity/differentiability properties of functions \ref{assump:kisynski_regularity} than the existence of uniform relative bounds for the operators~\ref{assump:simon_normequiv}, and we provided proof that the latter implies the~former.

{J.~Kisyński gave necessary conditions for the existence of solutions of the non-autonomous} Schrödinger equation for the case in which $\dom H(t)^{\sfrac{1}{2}}$ has a nontrivial dependence in time. This situation cannot be tackled by the framework proposed by B.~Simon. In~the case of a constant form domain, the necessary conditions imposed by J.~Kisyński are more restrictive, as~he required the continuity of the first derivatives, as~opposed to B.~Simon, who just required uniformly bounded derivatives. However, B.~Simon only gave necessary conditions to find solutions of the weak Schrödinger equation and was able to prove the continuity of the solutions, but not differentiability in the sense of $\mathcal{H}$. However, the~weak Schrödinger equation was not as physically meaningful as the strong Schrödinger equation. J.~Kisyński's extra regularity conditions allowed him to prove not only the differentiability of the solutions in the sense of $\mathcal{H}$, but also in the sense of $\mathcal{H}^+$. These imply that these solutions are solutions of the strong Schrödinger~equation.

The analysis presented in this article was devoted to the linear Schrödinger equation and the techniques used originated in the theory of linear operators in Banach and Hilbert spaces. Whether these techniques can be applied to treat also the nonlinear Schrödinger equation is an interesting perspective.

%%%%%%%%%%%%%%%%%%%%%%%%%%%%%%%%%%%%%%%%%%
\vspace{6pt}

%%%%%%%%%%%%%%%%%%%%%%%%%%%%%%%%%%%%%%%%%%

\section*{Acknowledgements}
A.B. and J.M.P.-P. acknowledge support provided by the ``Ministerio de Ciencia e Innovación'' Research Project PID2020-117477GB-I00, by the~QUITEMAD Project P2018/TCS-4342 funded by Madrid Government (Comunidad de Madrid-Spain) and by the Madrid Government (Comunidad de Madrid-Spain) under the Multiannual Agreement with UC3M in the line of ``Research Funds for Beatriz Galindo Fellowships'' (C\&QIG-BG-CM-UC3M), and~in the context of the V PRICIT (Regional Programme of Research and Technological Innovation).
  A.B. acknowledges financial support by ``Universidad Carlos III de Madrid'' through Ph.D. program grant PIPF UC3M 01-1819 and through its mobility grant in 2020.
  D.L. was partially supported by ``Istituto Nazionale di Fisica Nucleare'' (INFN) through the project ``QUANTUM'', and the Italian National Group of Mathematical Physics (GNFM-INdAM).
A.B. and  J.M.P.P. thank the Faculty of Nuclear Sciences and Physical Engineering at the Czech Technical University in Prague for its hospitality and David Krejčiřík for its support. 
 D.L. thanks the Department of Mathematics at ``Universidad Carlos III de Madrid'' for its hospitality.

\end{document}